\documentclass[11pt, leqno]{amsart}

\usepackage{amssymb}
\usepackage{amsmath}
\usepackage{amscd}
\usepackage{float}
\usepackage{graphicx}
\usepackage{amsfonts}
\usepackage{color}
\newtheorem{definition}{Definition}
\newtheorem{conjecture}{Conjecture}
\newtheorem{theorem}{Theorem}

\newtheorem{prop}{Proposition}

\newtheorem{rem}{Remark}

\def\CC{\mathbf C}
\def\LL{\mathcal L}

\def\FF{\mathcal F}
\def\A{\mathsf{A}}
\def\B{\mathsf{B}}
\def\C{\mathsf{C}}

\def\x{\mathsf{x}}

\def\L{\mathsf{L}}
\def\OO{\mathsf{O}}
\def\I{\mathbf{I}}
\def\K{\mathsf{K}}
\def\H{\mathsf{H}}
\def\X{\mathsf{X}}
\def\Y{\mathsf{Y}}
\def\F{\mathbf{F}}
\def\G{\mathsf{G}}
\def\Q{\mathsf{Q}}
\def\R{\mathsf{R}}
\def\S{\mathsf{S}}
\begin{document}

\title{NEW INVARIANTS OF LINKS  FROM A SKEIN INVARIANT OF COLORED LINKS }

\author{Francesca Aicardi }
 \address{The  Abdus Salam  Centre for Theoretical Physics (ICTP),  Strada  Costiera, 11,   34151  Trieste, Italy.}
 \email{faicardi@ictp.it}

\keywords{knot theory, knots invariants}

\subjclass{57M25 }

\date{}

 \begin{abstract}
New invariants  of  links are    constructed by  using  the skein invariant polynomial  of colored links  defined  in \cite{A1}. These  invariants  are  stronger  than  the  homflypt polynomial.
\end{abstract}

 \maketitle
\section*{Introduction}
 This  note refers to  the  work \cite{A1}, in  which  I have  defined,  via  skein  relation,  an  invariant of  colored  links.   Let me recall  some  definitions.

 An  oriented  link   with  $n$  components is  a  set  of $n$  disjoint  smooth oriented  closed   curves embedded  in $S^3$.  An oriented  link  is  {\sl  colored},  if each  component  is  provided  with  a  color. In other  words,  a  link is  colored  if a  function $\gamma$ is  defined  on the  set  of  components    to  a  finite   set  $H$  of  colors. Let's call {\sl  coloration}  such a   function $\gamma$.   Every  coloration introduces an  equivalence relation  in  the  set  of  components: two  components  are  equivalent  if  they  have  the  same  color.   Let $\CC$  be  the  set  of  components  of a  link $L$.  Two  colorations $\gamma$  and  $\gamma'$  of  $L$    are  said {\sl  equivalent} if  there is  a  bijection in $H$  between  the   images   $\gamma(\CC)$  and  $\gamma'(\CC)$.  Two    colorations  define   the same  partition  of  the  set  $\CC$  into  classes  of equivalence if  and  only  if  they  are  equivalent.

An invariant of colored links takes the same value on isotopic links with the  same  coloration, and may take different values on the same link with different colorations.
The   invariant $F$ of  colored links defined in  \cite{A1}  takes the  same value on  isotopic  links  with    equivalent  colorations.  In  Section 1,    the  statement  of the  theorem defining  the  invariant $F$ is  reported.

The invariant  $F$  has  three  variables,  say  $x$,  $w$  and  $t$.   There  are  two natural ways to define  an  invariant  of classical  links  by  using  $F$. The  first  way,
is  to  consider  a  link   as  a colored  link, whose  components  have all  the  same  color. In  this  case  the  invariant  $F$  becomes  independent of  $x$, and is  in fact  another  instance  of  the  homflypt  polynomial. In  particular,   it  becomes  the Jones  polynomial  in  the  variable  $t$  setting  $w=t^{1/2}$. Given a classical  link $L$, we  will  denote  $F^1(L)$  the  invariant  $F$ on  the  colored  link obtained  by  coloring  $L$  with a  sole  color.

The second  way,  is to  consider a link  as a  colored link, whose components  have  all  different  colors.  In  this  case  $F$ yields  an  invariant  which is  stronger  than the  homflypt polynomial. Given a classical  link $L_n$  with  $n$ components,  we  will  denote  $F^n(L_n)$  the  invariant  $F$ on  the  colored  link obtained  by  coloring  $L_n$  with  n colors.

Among  the  89 pairs of   non-isotopic links with  at  most  11 crossings, non  distinguished  by  the homflypt polynomial, 49  pairs  contain links with  three  components.  I  have  computed $F^3$  on  20 pairs of  such pairs:  six of  these  pairs  are distinguished  by  $F^3$ (see Proposition \ref{p2} in Section 3).

However,  there  are   other  ways  to  define  an  invariant  of  classical  links by  using $F$.  Indeed,  we  can  calculate  $F$  on  the  whole  set of    colored  links  obtained  from  a  link $L$ with $n$ components by  coloring it  by  all the  possible  non-equivalent  colorations. Then,  for  every  integer partition $p$ of  $n$ in $c$ parts ($c\le n$), we take   all  partitions  of  the  set  of $n$ components  in  $c$ blocks,  whose  cardinalities    correspond  to  the  partition $p$. These partitions  of  the set  of  components  are  said to be of type  $p$.     We thus  define the invariant  $\F_p^c$    as  the  set  of  values taken by  $F$ on the colored  links   obtained  by  coloring $L$ with  $c$ colors according to each one  of the type $p$ partitions, i.e.,   by coloring   the  components  with the  same  color if and only if  they  belong to  the same  block. Example:  the  link  below  has  three  components,  that  are colored according  to  the  partitions indicated.

\centerline{\includegraphics{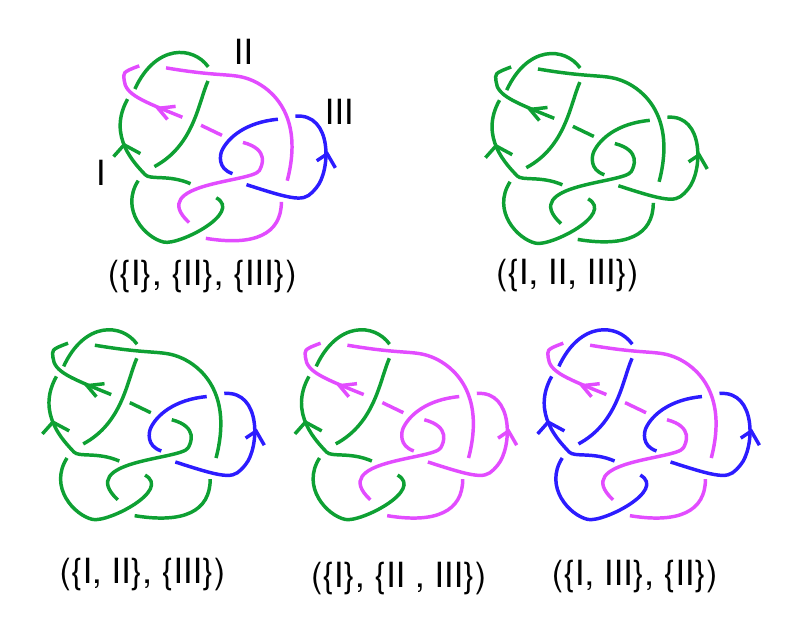}}

 The  invariant $\F^2_{(2,1)}$  distinguishes    pairs of  the above mentioned  pairs of  non  isotopic  links, which  are  not  distinguished by  $F^3$.

As  we  have  said  in \cite{A1}, the definition  of  $F$ is a  byproduct  of  the  study of another  class  of  links,  the  {\sl  tied  links} (see \cite{AJTL}).  In fact,  the  tied  links  are links  whose components  may  be   connected  in  several  points  by  ties.  But  these  ties behave   exactly  as  an  equivalence relation  between  the  link  components,  so  that   isotopic  tied  links are  isotopic  colored links with  equivalent  colorations (see Section 1 for more precise  definitions).  Tied  links  can  be  seen also  as  the  closure of  tied  braids,  that constitute  a  diagrammatic  representation  of  the  elements of an  abstract  algebra,  the  so-called  nowadays  bt-algebra (algebra  of  braids  and  ties).   In  \cite{AJTL} we  have  defined,  via  skein relation, a   polynomial  $\FF$,  invariant of  tied links, and we  have shown that $\FF$ can  be realized also  by  the Jones recipe,  using  the  trace on the  bt-algebra  established in \cite{AJtrace}. The  properties of  the  trace  allowed  us  to make  a computer program  to  calculate  $\FF$.   We  may  define a  tied  link $TL$ from a  colored  link $L$  so  that  $F(L)$ is  recovered from  $\FF(TL)$  by  a  variable change.  This has allowed me to  verify  the  values of  $F$ obtained on  colored  links  by  the  skein  relation,  presented  in  this  paper.

Finally, I  show  some  unexpected  identities  satisfied  from  the  values  forming the  invariants $F^c_p$  for  the  pairs  distinguished  by such invariants,  that  allows to  formulate a conjecture.  The calculation  of  $F^c_p$  on  all  pairs of  three-component  links  not distinguished  by  the  homflypt  polynomial  has to   be  completed in  order  to  understand  these identities  or to disprove the conjecture.

\section{The  invariant  $F$ of  colored  links}

\begin{definition}\label{isotop} \rm Two oriented colored links    are  {\sl c-isotopic}  if
  they  are  ambient  isotopic,  and  their  colorations  are  equivalent.
\end{definition}

A  colored link diagram  is  like the  diagram of a  link,  where  each component  is  colored.

Let  $R$  be  a  commutative  ring and  $\LL_c$  be  the set of   colored oriented link  diagrams.
By   {\sl invariant of  colored  links} we  mean     a  function $I:  \LL_c \rightarrow R$, which is  constant     on each class  of c-isotopic links.

 \begin{rem}\rm  \label{rem1} In  the  sequel, we denote by   $L$     an  oriented   colored  link  as  well as  its  diagram,  if  there  is  no risk  of  confusion. \end{rem}


 \begin{theorem} \label{the1} There  exists  a rational function  in  the  variables $x,w,t$,   $F:\LL_c \rightarrow   \mathbb{Q}(x,t,w)$,     invariant  of oriented colored  links,    uniquely  defined  by  the   following  three  conditions  on  colored-link  diagrams:
\begin{itemize}
\item[I] The  value of  $F$ is  equal to 1  on  the  unknotted circle.

\item[II] Let  $L$   be  a colored link,  with  $n$  components  and  $c$  colors.  By  $[L;\OO]$  we  denote  the  colored link  with  $n+1$ components consisting of $L$  and  the unknotted and unlinked  circle, colored  with  the $(c+1)$-th     color.
Then

\[  F([L;\OO])=  \frac {1}{wx} F(L). \]

\item[III](skein  relation)  Let $L_+$  and   $L_-$   be  the  diagrams  of two colored  links,  that are  the  same outside    a  small  disc into  which two  strands  enter,   and  inside  this  disc  look  as
   shown  in figure \ref{Linv1}a. Black and gray  indicate  any two  colors,  non  necessarily  distinct.  Similarly,  let   $L_\sim$  and   $L_{+,\sim}$  be two links  that  inside  the  disc look  as  shown  in  figure \ref{Linv1}b,  while outside  the  disc  coincide with $L_+$  and  $L_-$, except -possibly-  for  the  colors of all the components of the  link having  the same colors as    the  strands entering  the  disc:  all these  components   in   $L_\sim$  and   $L_{+,\sim}$  are  colored by  a  sole  color. Then  the  following  identity  holds:
 \[ \frac{1}{  w}F( L_ +)-w F (L_-) = \left(1-t^{-1}  \right) F (L_\sim) + \frac{1}{w} (1-t^{-1} ) F(L_{+,\sim}).\]

\end{itemize}
\end{theorem}

\begin{figure}[H]
\centering
\includegraphics{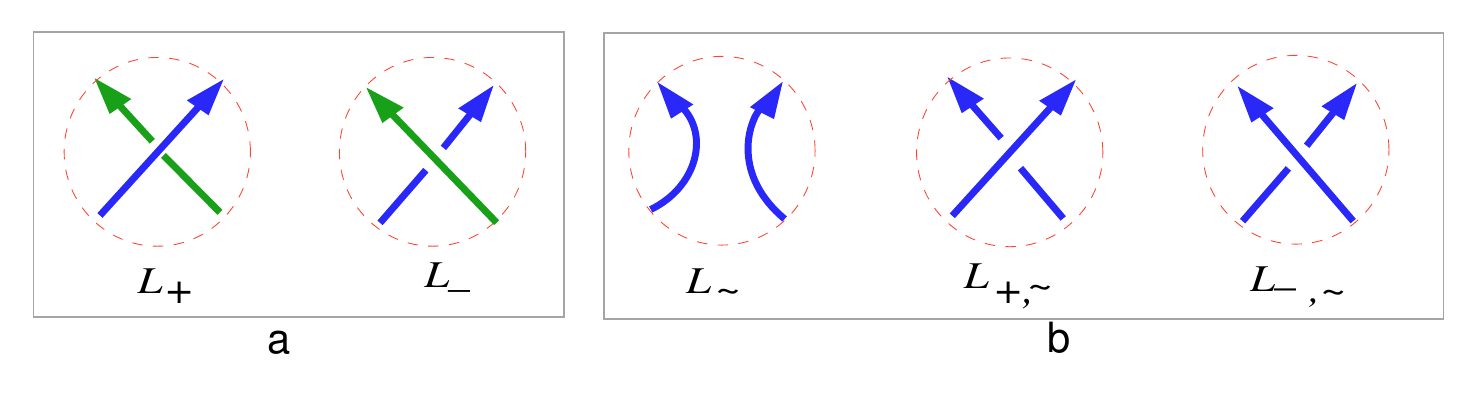}
\caption{a) The  discs  where   $L_+$  and $L_-$  are  not c-isotopic. b) The  discs where  $L_{\sim}$,   $L_{+,\sim}$,  and $L_{-,\sim}$   are not c-isotopic. Blue and  green   indicate  any  two   colors. }\label{Linv1}
\end{figure}

\begin{rem}\label{rem2}\rm
The skein relation  III  holds  for  {\sl any}  two  colors  of  the  involved   strands. In  particular, if  the  two  colors coincide, then  $L_+=L_{+,\sim}$  and  $L_-=L_{-,\sim}$, so  that relation III  is  reduced  to  the relation

 \begin{itemize}
 \item[IV]
 \[ \frac{1}{t w }F( L_ {+,\sim})- w F (L_{-,\sim}) = (1-t^{-1}) F (L_\sim).\]
\end{itemize}
 Relation III also implies   following  two  relations; they  will  be  used in the  sequel.

\begin{itemize}
\item[Va]
 \[ \frac{1}{ w }F( L_ {+})= w \left[F(L_-) +\left( {t}-1 \right) F(L_{-,\sim})\right] + \left(t-1\right) F (L_\sim).\]

\item[Vb]
\[   w  F( L_ {-})= \frac{1}{ w } [F(L_+) +(t^{-1}-1) F(L_{+,\sim})] + ({t^{-1}}-1) F (L_,\sim).\]
 \end{itemize}
\end{rem}

\begin{rem}\label{rem3} \rm As  an  example of using  relations I, II  and  III,  let's  calculate  the  value  of  $F$ on  the  unlink $\OO^c_n$, consisting of  $n$ circles  and  $c$ colors.
If  $n=c$, then,  using II, $$F(\OO^c_c)= \frac{1}{wx}F(\OO^{c-1}_{c-1})=\left(\frac{1}{wx} \right)^{c-1}.$$
If  $n>c$,   then there are two circles colored by the same color. Regard them as belonging to $L_{_{\displaystyle\widetilde{\ \ }}}$. Then $L_+=L_-=\OO_{n-1}^c$ (see  figure).

\centerline{\includegraphics{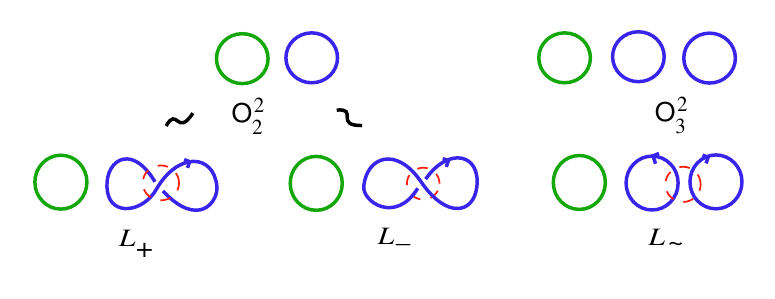}\label{Linv2}}

By IV,
$\displaystyle (1-t^{-1})F(\OO_n^c)=\Bigl(\frac1{tw}-w\Bigr)F(\OO_{n-1}^c)$.
So
$$F(\OO_n^c)=\frac{1-tw^2}{tw(1-t^{-1})}F(\OO_{n-1}^c) =
\frac{tw^2-1}{(1-t)w}F(\OO_{n-1}^c):= \frac{y}{xw}F(\OO_{n-1}^c),
$$
 where  \begin{equation}\label{formulay}  y=  \frac {x (t w^2-1)}{1-t}.\end{equation}
  By induction,
$\displaystyle F(\OO_n^c)=\Bigl(\frac{y}{xw}\Bigr)^{n-c}F(\OO_c^c)$.
 Thus
 \begin{equation}\label{formulaonc} F(\OO_n^c)=\frac{y^{n-c}}{(wx)^{n-1}} . \end{equation}

\end{rem}

\begin{rem} \label{rem4} \rm  Let  $L$  be  a  knot or   a  link whose components  have  all the  same  color. Then    $F(L)$  is  defined  by I and IV.   Recall  that  the  homflypt  polynomial $P$  is  defined  by  the condition  that  $P(\OO)=1$  and  the skein relation
 \begin{equation}\label{skeinrule}   \ell P(L_{+,\sim})  +  \ell^{-1} P(L_{-,\sim}) +  m  P(L_\sim)= 0, \end{equation}
 where $\ell$  and  $m$  are  two  complex  variables.
 Comparing the   skein  relation  IV,  multiplied  by  $ t^{1/2}$,  with the equation  above,  we see that   $F$      is  a  homflypt polynomial  with
     \[   \ell = i  \frac{1}{ w\sqrt{t}}   \quad \text{and} \quad m=i\left( \frac {1}{\sqrt t}-  \sqrt{t} \right). \]
     In  particular,  for  $w=t^{1/2}$  it  becomes  the  Jones  polynomial.
\end{rem}

\begin{rem} \label{rem5} \rm The  polynomial  $F$  is  multiplicative  with  respect to the  connected sum of  colored  links $L_1$  and  $L_2$ in  the  following  sense.  The  components $C_1$ and  $C_2$  of  the  two  links  that  merge  must  have the  same  color,  while  the  colors of  the components  in  $L_1$  not  equal  the  color  of  $C_1$ must  be  different from  the  colors  of  the  components  in $L_2$ not equal to the  color of  $C_2$.

\end{rem}

\section{The  new invariants of links}

Let  $L$  be  a  classical  link with $n$  components. We  denote by  $\CC_n:=\{C_1,C_2,\dots,C_n\}$  the  ordered  set of  the components of  $L$. A set-partition
 $$\I_j:=(I_1,I_2,\dots,I_{c_j})$$ of $\CC_n$,  contains  $c_j\le n$  non void blocks.  For  every  set-partition $\I_j$ of $\CC_n$,  we  define a  colored link $L_j$ in  this  way: we  color the  components of  $L$   with  $c_j$  colors,  in  such a  way  that  the  components  entering the  same  block  have  the  same  color.   Observe  that the  number $N$ of   non-equivalent colorations of the  link  $L$  satisfies  $N\le B_n$,  where $B_n$ is  the  $n$-th  Bell  number.   The  equality holds in  general except in   particular  cases  of symmetries among the  components of  the  link.

 \begin{definition}\rm Given  a  set partition $\I:= (I_1,I_2,\dots,I_{c})$ of  $\CC_n$, let  $n_i=|I_i|$, so  that $n_1+ n_2 +\dots+ n_c=n$.
   We  call the  integer partition $p$ of $n$ in $c$ part,
  \[  p:=(n_1,n_2,\dots, n_c)  \quad \quad  n_i\ge n_{i+1}\] the {\sl type}  of  the  partition $\I$.
\end{definition}

  For  instance,  if $n=4$, among  the  following partitions of $\CC_4$ in  two  blocks:
  \[ \I_1:=(\{C_1,C_2\},\{C_3,C_4\}), \quad  \I_2:=(\{C_1\},\{C_2,C_3,C_4\}), \quad   \I_3:=(\{C_1,C_4\},\{C_2,C_3\}),\]  the  partitions  $\I_1$   and  $\I_3$ have the  same  type $(2,2)$, and  $\I_2$ has  type $(3,1)$.

\begin{definition} \rm For  every partition $p$  of $n$ in $c$ parts   ($1\le c\le n$),   let $N^c_{p}$ be    the  number of partitions  of $\CC_n$  of  type $p$ which  define non-equivalent colorations  of  $L$  with $c$ colors,   and $L_1,\dots,L_{N^c_{p}}$ be the  corresponding  colored  links.  The  invariant $\F^c_p$ of $L$  is  the non-ordered   set
\[   \F^c_p(L):=\{ F(L_1), F(L_2), \dots, F(L_{N^c_{p}}) \}.\]
\end{definition}

\begin{rem}\label{rem6}\rm It is  evident  that for a  link $L$  wit  $n$  components,  if $c=1$, $p=(n)$, so $N^1_n=1$, and  $\F^1_{(n)}(L)=\{F^1(L)\}$, i.e., the  value  of $F$ on  the  link  $L$  with all  component  colored  with  the same  color.    Furthermore, for  $c=n$ the  invariant $\F^n_{\underbrace{\small{(1,1, \dots, 1)}}_n}= \{  F^n(L) \} $.
\end{rem}

\begin{rem}\label{rem7} \rm In  the  following  figures,  two  arrows going from  a  link $L$  with  a  marked  crossing, to two colored links $L_1$  and $L_2$,  indicate  that  $F(L)=\alpha F(L_1)+\beta F(L_2)$, where $\alpha$  and  $\beta$  are given by  relation  IV (Remark \ref{rem2});   three arrows going  from  a colored  link $L$  with a marked  crossing,  to  three links  $L_1$,  $L_2$ and  $L_3$, indicate  that  $F(L)=\gamma F(L_1)+\delta F(L_2)+\lambda F(L_3)$,  where  $\gamma$, $\delta$  and  $\lambda$  are  given  by  relations  Va  or   Vb, depending on  the sign of  the  marked crossing.
\end{rem}

\subsection{An example} We  give  here  the  experimental  evidence  that  for  $c>1$, $\F^c_p$  is more powerful than  the homflypt polynomial.

Consider,  for  instance,  the  links $A$  and  $B$, shown  in  \ref{Linv2}.  They  have  four components, six  positive  crossings,  and are  colored  with a  sole  color. $A$  and  $B$ are  evidently  non  isotopic.
The  calculation of  $F(A)$  and  $F(A)$   shows  that  $A$  and $B$   are  not  distinguished  by  the  homflypt polynomial (see  Remark \ref{rem4}). Indeed,
applying  the  skein relation IV  to the   crossings  marked  in  Figure  \ref{Linv3},   we obtain
   $$F(A):= F(A_{+,\sim})= tw^2 F(A_{-,\sim})+ (t-1)w F(A_\sim), $$   $$F(B):= F(B_{+,\sim})= tw^2 F(B{-,\sim})+ (t-1)w F(B_\sim).$$
   But $A_{-,\sim}=B_{-,\sim}=L_1$  and  $A_\sim=B_\sim=L_2$.  Therefore  $F$ coincides on  $A$ and on $B$.

\begin{figure}[H]
\centering
\includegraphics{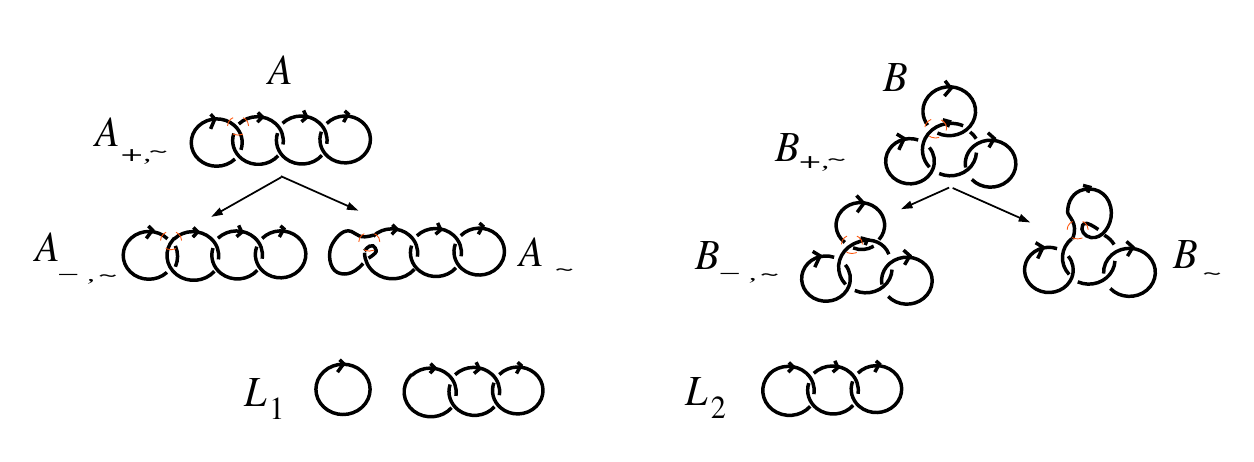}
\caption{  $F(A)$ and   $F(B)$  coincide}\label{Linv3}
\end{figure}

Consider now  the colored links  obtained  by  coloring $A$ and $B$ with two  colors, more precisely three components with one  color  and the  remaining   component  by  another  color (i.e.,  corresponding  to  the  partitions of  type   $p=(3,1)$).  There  are  two of  such  non-equivalent colorations of $A$  and  of $B$,  giving  the  colored  links  $A_1$, $A_2$  and  $B_1$,  $B_2$ respectively  (see  Figure \ref{Linv4})
\begin{figure}[H]
\centering
\includegraphics{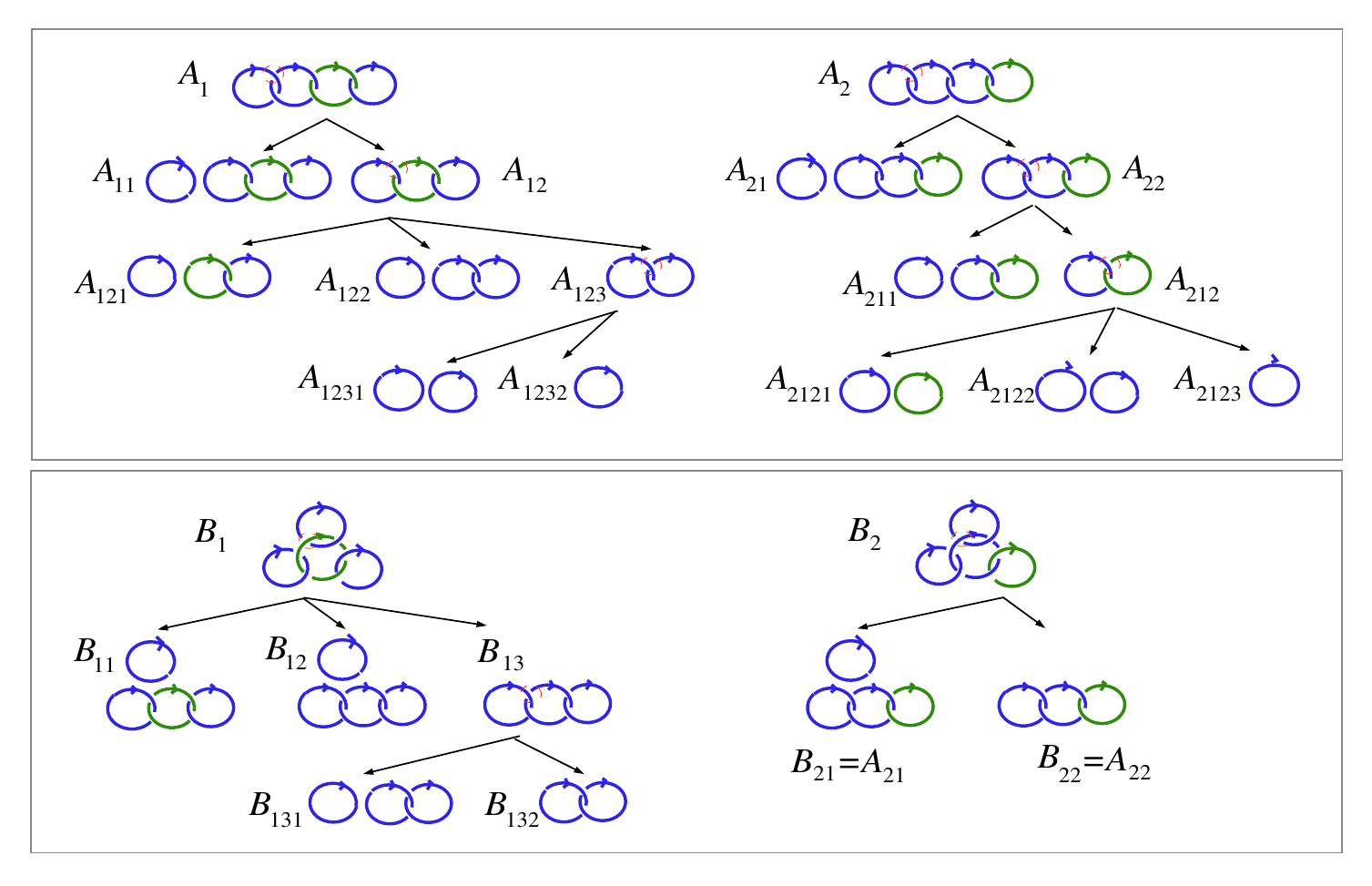}
\caption{Calculation of $\F^2_{(3,1)}(A)$ and   $\F^2_{(3,1)}(B)$}\label{Linv4}
\end{figure}

\begin{prop} The  links $A$ and  $B$  are  distinguished  by  the  invariant $\F^2_{3,1}$.
\end{prop}

\begin{proof}
We  have $\F^2_{(3,1)}(A)=\{F(A_1),F(A_2)\}$  and  $\F^2_{(3,1)}(B)=\{F(B_1),F(B_2)\}.$  From  Figure  \ref{Linv4}  we  see  that  $F(B_2)=F(A_2)$.  We  calculate  now  $F(A_1)$ and  $F(B_1)$.

We  have:\\
$F(A_1)=w^2t F(A_{11})+w(t-1)F(A_{12})$,\\
$F(A_{11})=F(A_{12})y/(wx)$,   $\qquad F(A_{12})=w^2 F(A_{ 121})+w^2(t-1)F(A_{122})+w(t-1)F(A_{123})$,\\
$F(A_{121})=F(A_{212})y/(wx)$,  $\quad F(A_{122})=F(A_{123})y/(wx)$,\\
$F(A_{123})= w^2t F(A_{1231}+w(t-1) F(A_{1232}),$\\
 $F(A_{212})=w^2 F(A_{2121})+w^2(t-1)F(A_{2122})+w(t-1)F(A_{2123})$\\
$F(A_{2121})=1/(wx))$,
 $\qquad F(A_{1231}=F(A_{2122})=y/(wx)$,  $\qquad F(A_{1232})=F(A_{2123})=1$. \\
We  get \\
$F(A_1)= (1-x t^5 w^6-t-t w^2+t^2-2 x t w^2+2 x t+w^4 t^2 x-t^3 w^2+t^3 w^4
+2 x t^2 w^2 - 6 x t^3 w^2 \\
-3 x t^2+4 x t^3+2 x t^3 w^4-x t^4 w^6+3 x t^4 w^2+3 x t^5 w^4-3 x t^5 w^2
-2 x t^4+x t^5) w^3/(x (t-1)^2).$

On  the  other hand,  we  have\\
$F(B_1)= w^2 F(B_{11})+w^2(t-1)F(B_{12})+w(t-1)F(B_{13}),$\\
$F(B_{11})=F(A_{11}),$ $\quad F(B_{12})=F(B_{13}) y/(wx),$
$\quad F(B_{13})=w^2t F(B_{131})+w(t-1)F(B_{132})$,\\
$F(B_{131})=F(B_{132})y/(wx)$, $\qquad F(B_{132})=F(A_{123})$.\\
Therefore,\\
$F(B_1)= w^3(w^4 t^2+3 w^4 t^2 x-6 x t^3 w^2-3 x t^2+4 x t^3-2 x t^4-w^6 t^3 x-x t^4 w^6+3 x t^4 w^2-x t^5 w^6\\
 +3 x t^5 w^4-3 x t^5 w^2+x t^5+3 x t^3 w^4-3 x t w^2+3 x t+1-2 t w^2)/(x (t-1)^2).$

Since  $F(B_1)\not=F(A_1)$,  the  links  $A$  and $B$  are  distinguished  by  $\F^2_{(3,1)}$.
\end{proof}

\section{Non  trivial  examples}
In  this  section  we show  some pairs of links which  are  not  distinguished  by  the  homflypt polynomial,  but  that  are  distinguished  by $\F^c_p$.

The  calculation  of  $F$  by  the skein  relation  is  shown  in  the  figures  for  some  pairs. For  the  meaning of  the  arrows,  see  Remark  \ref{rem7}.  The  skein-trees  terminate  in some  simple  basic  colored links, belonging to  a list   $\L1,\L2,\dots,\L41$,  shown  in  Figure \ref{List}.  The  values  of  $F$  on  these colored links  are  given  in  Appendix 1.

The  links are  named according  to LinkInfo,  see \cite{ChaLi}.

\begin{prop}\label{p1} The invariant $\F^3_{(1,1,1)}$  distinguishes the following pairs  of links:

\[\begin{array} {|c|c|}

\hline
 L11n325 \{1,1\}   &  L11n424 \{0,0\}   \\ \hline
L11n356 \{1,0\} & L11n434\{0,0\}  \\ \hline
L10n79\{1,1\} & L10n95\{1,0\}  \\ \hline
L11a404\{1,1\} & L11a428\{0,1\} \\ \hline
L10n76\{1,1\} & L11n425\{1,0\} \\ \hline
L11n358\{1,1\} & L11n418\{1,0\} \\ \hline
\end{array}
\]

\end{prop}

\begin{proof}
These  links  have  three  components.  The values  of  $F$  are  different on the  colored   links,  obtained  by coloring  each of these  links  with  three  different  colors. In   Appendix 2, the values of $F^3$  are listed  for  the  pairs,  except  for the pair ($L10n76\{1,1\}$, $L11n425\{1,0\}$),  for  which  the  invariants  are  calculated  in  Section 4
and are shown in  Appendix 4.

\end{proof}

\begin{prop}\label{p2} The invariant $\F^2_{(2,1)}$  distinguishes  the following pairs  of links, which  are  not   distinguished by  $\F^3_{(1,1,1)}$:
\[\begin{array} {|c|c|}

\hline
 L11n358\{0,1\}  &   L11n418 \{0,0\} \\ \hline
L11a467 \{0,1\}  &   L11a527 \{0,0\} \\ \hline
 \end{array}
\]
\end{prop}

 \begin{figure}[H]
\centering
\includegraphics{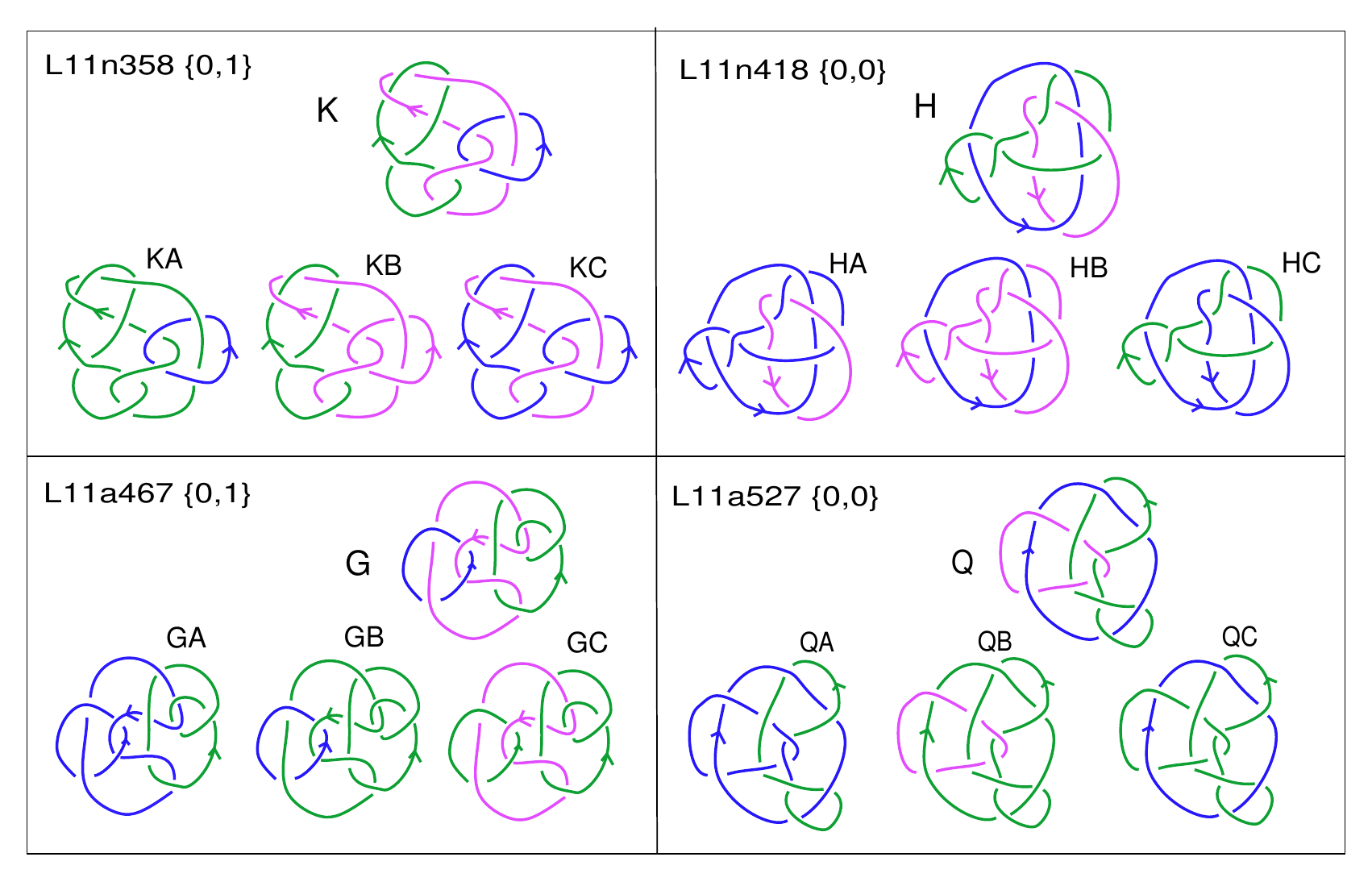}
\caption{Colored links for  the  calculation of   $\F^2_{(2,1)}$}  \label{pairs}
\end{figure}

\begin{proof}

  For  every  one of  the  above pairs, we  proceed   in  this  way: we firstly calculate  $F^3$ on  the  links  of  the  pair  (all  links  have  three components). These  values  coincide in each pair,  so  we  obtain that   $\F^3_{(1,1,1)}$  does not  distinguish  these  links.
 Then  we  calculate    $\F^2_{(2,1)}$.

We  show  here the  calculation of $\F^2_{(2,1)}$  on  the links of the pair  ($\K,\H$).

There  are  three non-equivalent  colorations  of $\K$  and  $\H$ corresponding  to  the  partitions  of type  $(2,1)$ of the  set  of  components.  The  non c-isotopic colored  links obtained  by  such  colorations  are  called  respectively $\K\A, \K\B, \K\C$,  and  $\H\A,\H\B,\H\C$, see Figure \ref{pairs}.

\begin{figure}[H]
\centering
\includegraphics{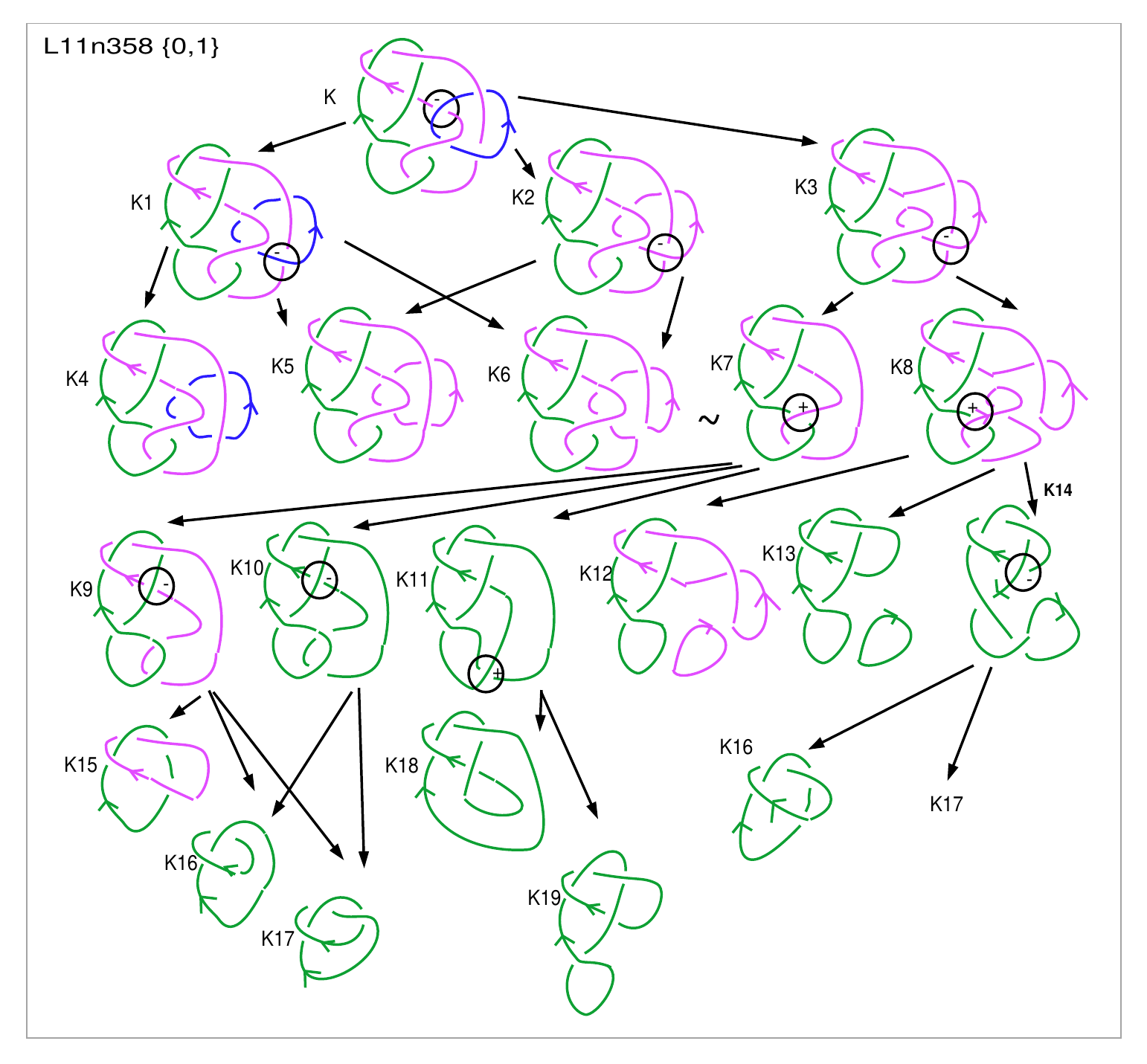}
\caption{Calculation of $F^3(\K)$}  \label{LK}
\end{figure}

\begin{figure}[H]
\centering
\includegraphics{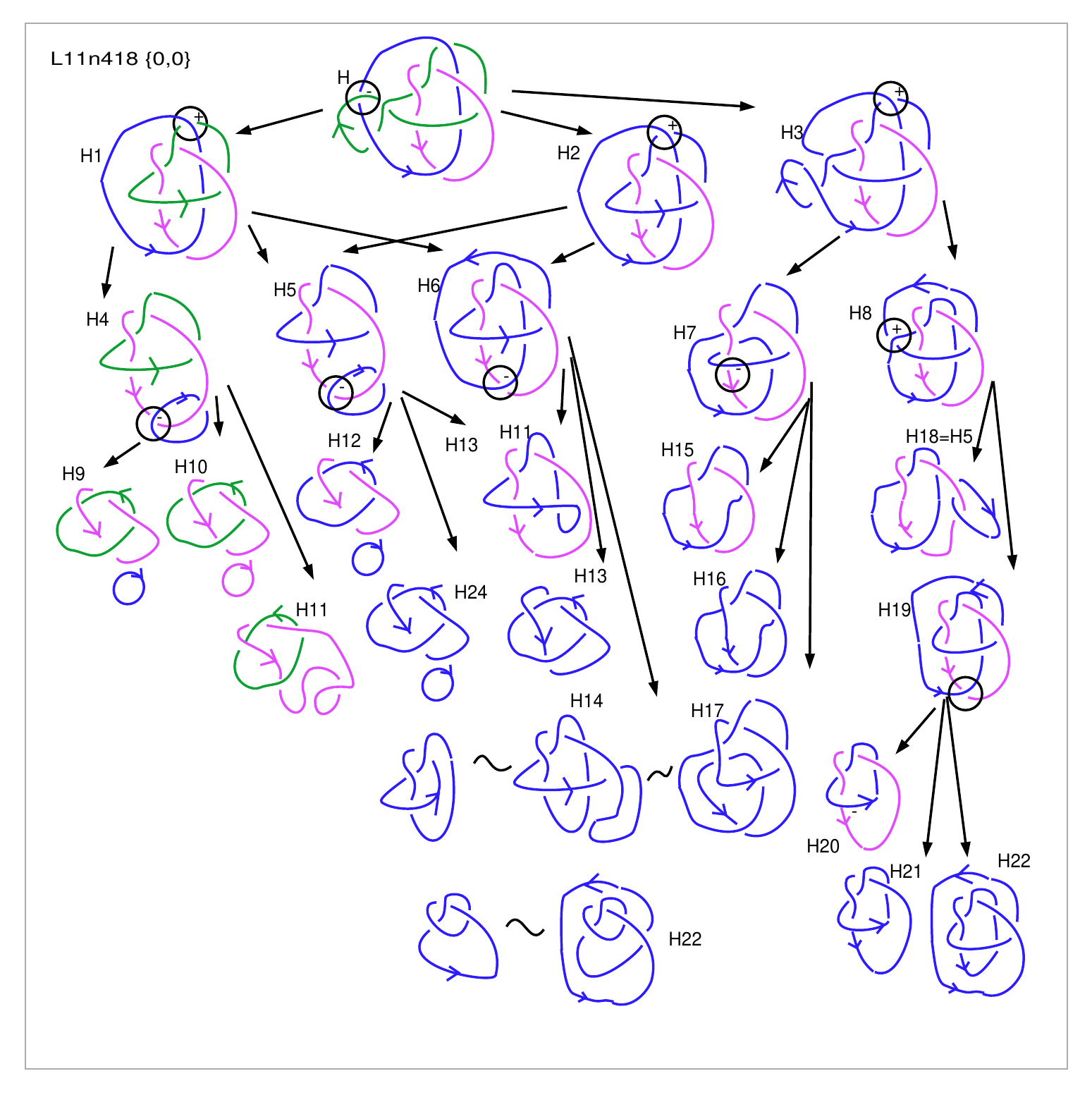}
\caption{Calculation of $F^3(\H)$} \label{LH}
\end{figure}

From Figures \ref{LK} and \ref{LH}  we obtain  the  values  $F^3(\K)=F^3(\H)$,
see  Appendix 3.

We  calculate now   $\F^2_{(2,1)}(\K)=\{ F(\K\A), F(\K\B), F(\K\C)  \}$, see Figure \ref{KABC}. Then   we  compare  this  set  of  values  with
$\F^2_{(2,1)}(\H)=\{ F(\H\A), F(\H\B), F(\H\C) \}$, obtained  from Figure \ref{HABC}.
These sets of  values, shown in   Appendix 3,
 are  different.

\end{proof}

\section{Identities}

For  a  link  with  three components, one  could define,  instead  of     $\F^2_{(2,1)}$,  the invariant  defined as the   {\sl  sum} of  the values  belonging to   the  set $\F^2_{(2,1)}$. Surprisingly,  despite  the  fact  that  the  values  are  all  different  for  the  pairs of  Proposition \ref{p2},  their  sums  coincide.  We  have indeed:

$$  F(\K\A)+ F(\K\B)+ F(\K\C)- F(\H\A)- F(\H\B) - F(\H\C)=0 $$

 $$  F(\G\A)+ F(\G\B)+ F(\G\C)- F(\Q\A)- F(\Q\B) - F(\Q\C)=0 $$

Furthermore,  I  have  found  that,  in  the  pairs  of  Proposition \ref{p2},
the links  are  distinguished  not  only  by $\F^3_{1,1,1}$,  but  by  $\F^2_{2,1}$  as  well.   Naming  $\X$ and $\Y$  the  colored links  obtained  from  a  pair  by  coloring  the  three component  with  three different colors,  and  $\X\A$,  $\X\B$, $\X\C$, $\Y\A$,  $\Y\B$, $\Y\C$,   the  colored links obtained from  that pair  by  coloring the  components  with  two different  colors,  the  corresponding  values  of $F$ satisfy:

\[ F(\X\A)+ F(\X\B)+ F(\X\C)- F(\Y\A)- F(\Y\B) - F(\Y\C)= F(\X) - F(\Y). \]

Therefore,  in  the  case  of  three  components,   we  formulate   the  following   conjecture. Observe  that  in  all  the   pairs of links here  considered, the  homflypt polynomial, and  then $\F^1_{(3)}$,  coincide.  Denote by  $\Sigma^c$  the  sum of  the  elements of  the  set  $\F^c_p$.

\begin{conjecture} Let  $L_1$ and  $L_2$ be two  links  with  three  components. If    $\Sigma^1(L_1)=\Sigma^1(L_2)$,  then
  \[  \Sigma^3 (L_1)- \Sigma^2(L_1) =  \Sigma^3 (L_2)- \Sigma^2(L_2).\]
\end{conjecture}

The  reasons  of  the above identities  are  not  evident to me  and deserve    further investigations. Also,  the  invariants  $\F^3_p$   have to  be  calculated  on  all  pairs of  non isotopic  links   with  the  same  value of  $F^1$.

{\bf Example.}       In  Figures \ref{RS}, \ref{RABC} and \ref{SABC}   the  trees for the calculations  of $\F^3_{(1,1,1)}$  and  $\F^2_{(2,1)}$ on  the  links  $\R=L10n76\{1,1\}$ and  $\S=L11n425\{1,0\}$ are  shown.
Observe  that  the  values of  $F$ on  some  leaves of  the  skein-trees for  the  colored  links  obtained  from $\R$ and  $\S$,   are  obtained by using   Remark \ref{rem3},  e.g.
 $\R1=\L39/(wz)$,   $\R2=\L35/(wx)$.  Also,   Remark \ref{rem5} is used  to  calculate some  values of $F$,  e.g.
$F(\S4)= F(\L39)F(\L5)$,  $F(\S\B5)=F(\L35)F(\L4)$.

The values  of  $\F^3$  and  $\F^2_{(2,1)}$  are  shown  in  Appendix 4.

{\bf Acknowledgement.}  A particular  thank  to  Konstantinos  Karvounis  for  having providing  me with the  list of pairs of non  isotopic links that  are  not  distinguished  by  the  homflypt polynomial.

\begin{figure}[H]
\centering
\includegraphics{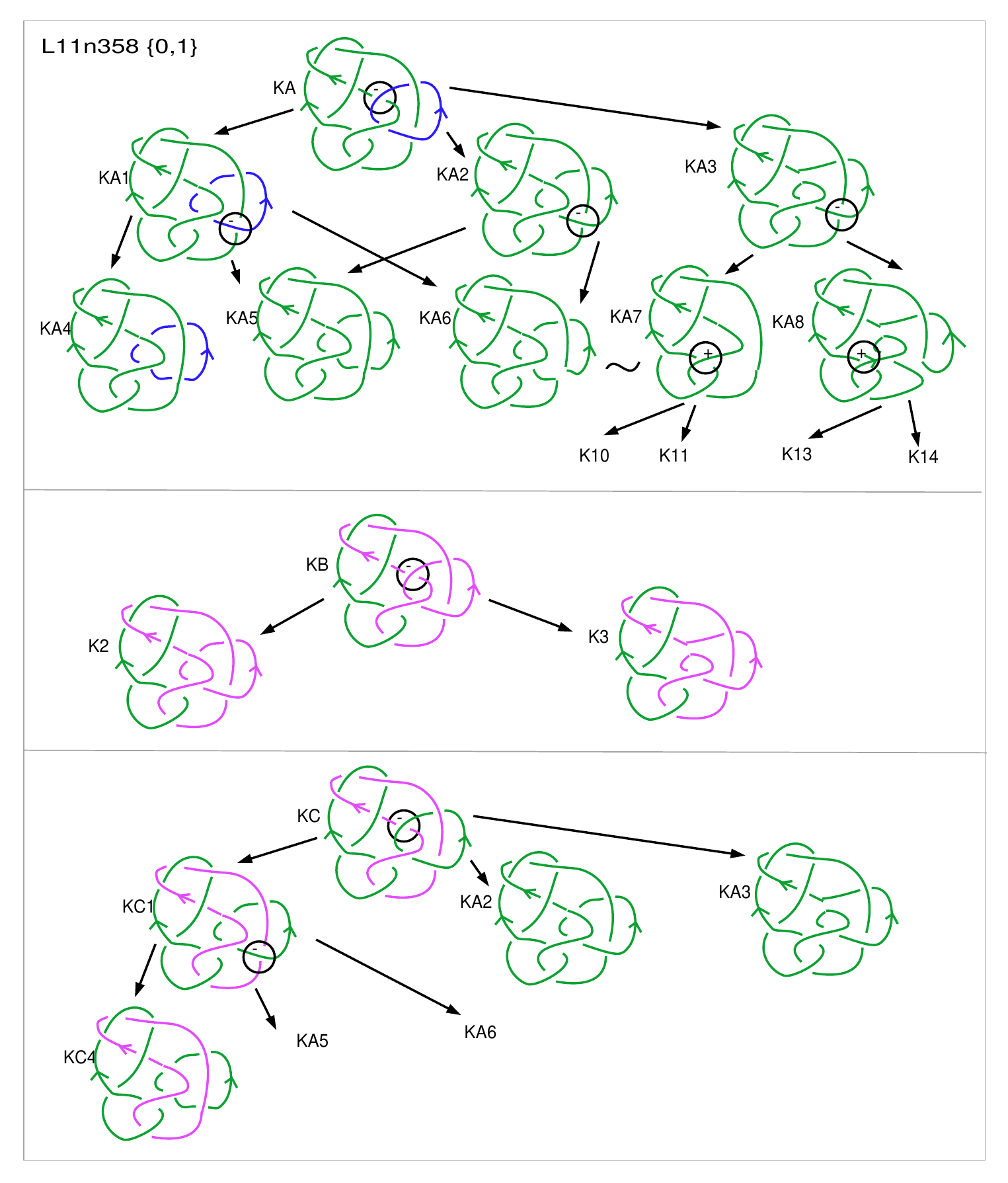}
\caption{Calculation of $F(\K\A)$, $F(\K\B)$ and $F(\K\C)$ }\label{KABC}
\end{figure}

\begin{figure}[H]
\centering
\includegraphics{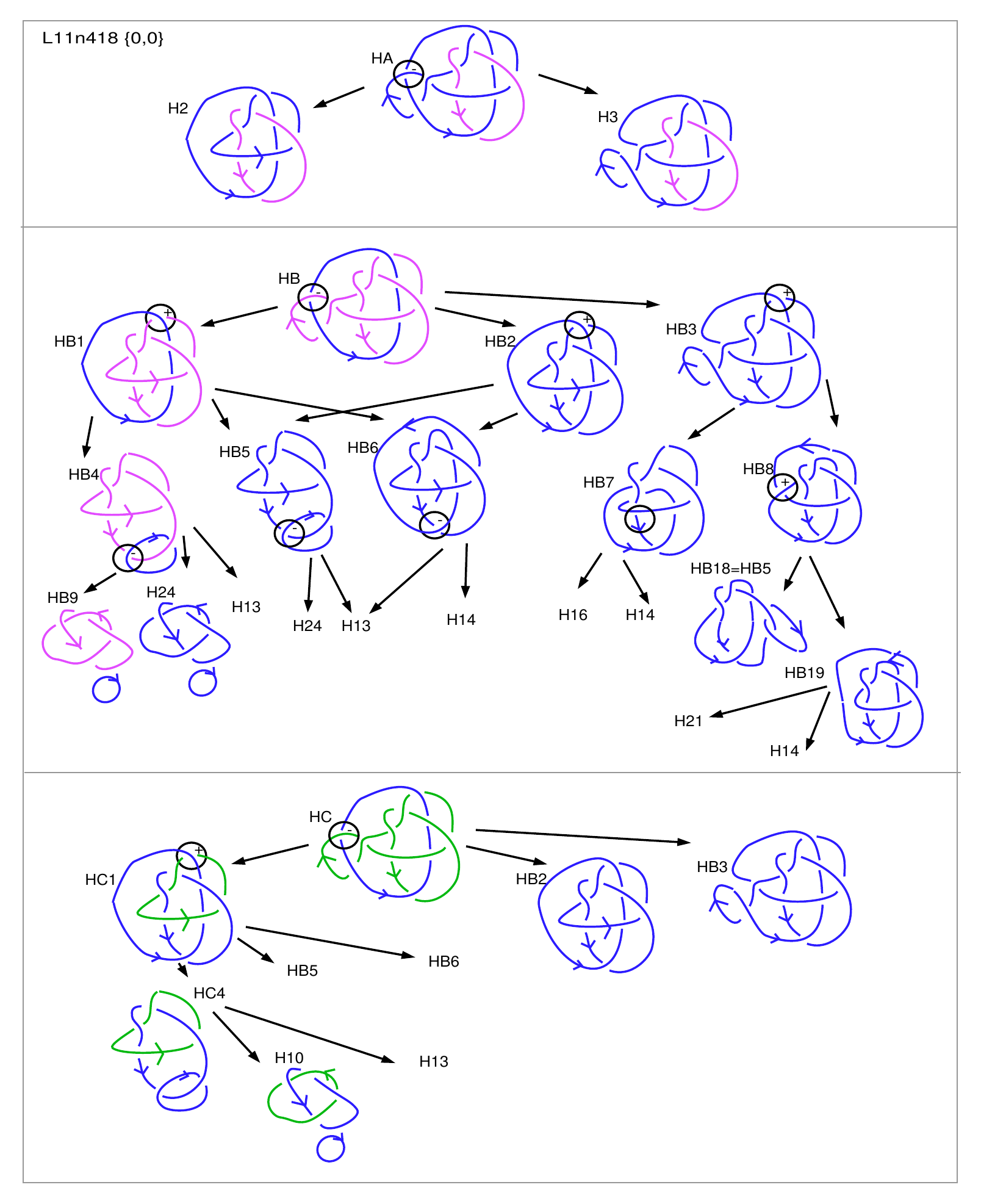}
\caption{Calculation of $F(\H\A)$, $F(\H\B)$ and $F(\H\C)$   }\label{HABC}
\end{figure}

\begin{figure}[H]
\centering
\includegraphics{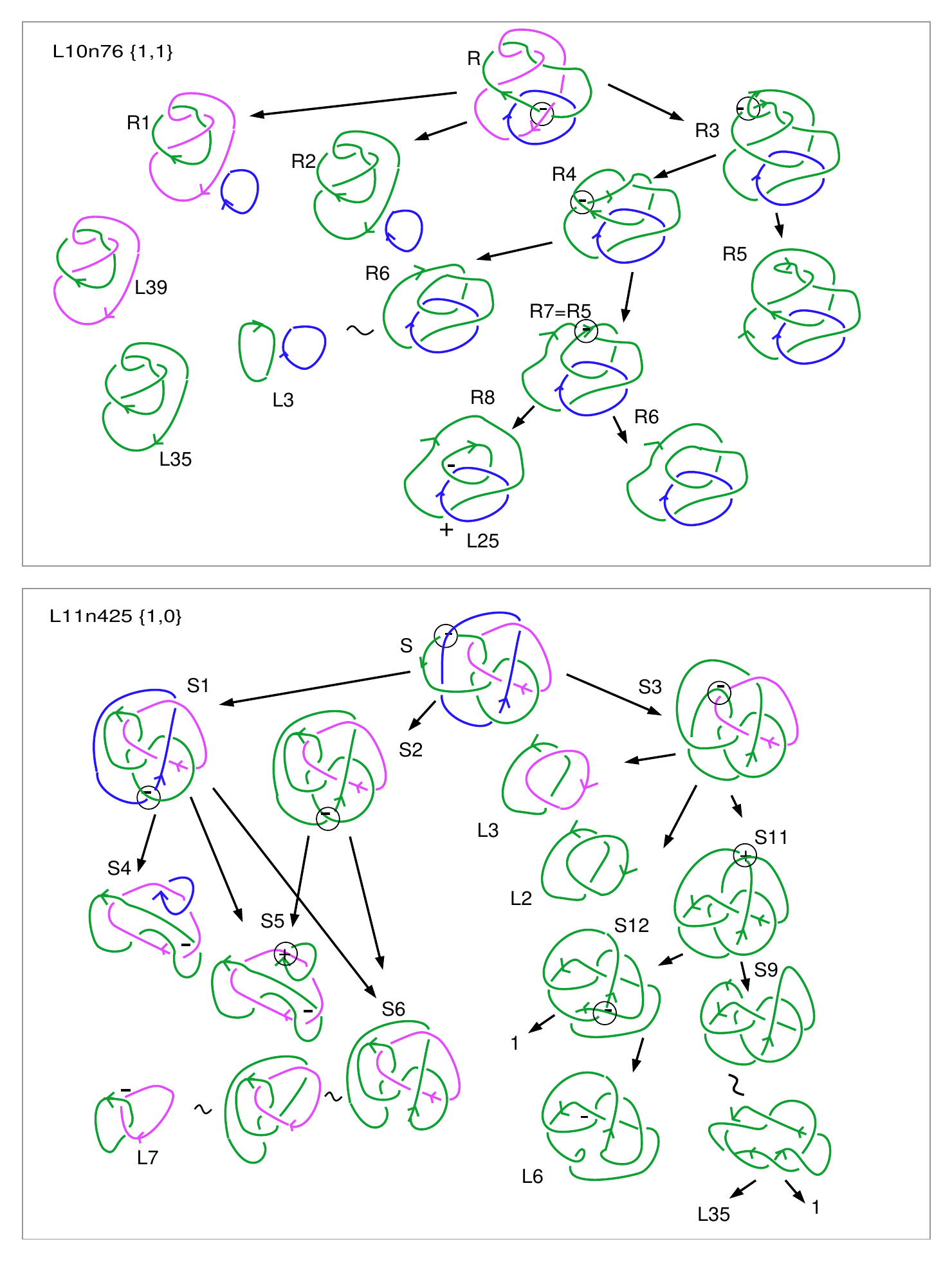}
\caption{Calculation of $F^3$ on the  links  $\R$  and  $\S$ }\label{RS}
\end{figure}

\begin{figure}[H]
\centering
\includegraphics{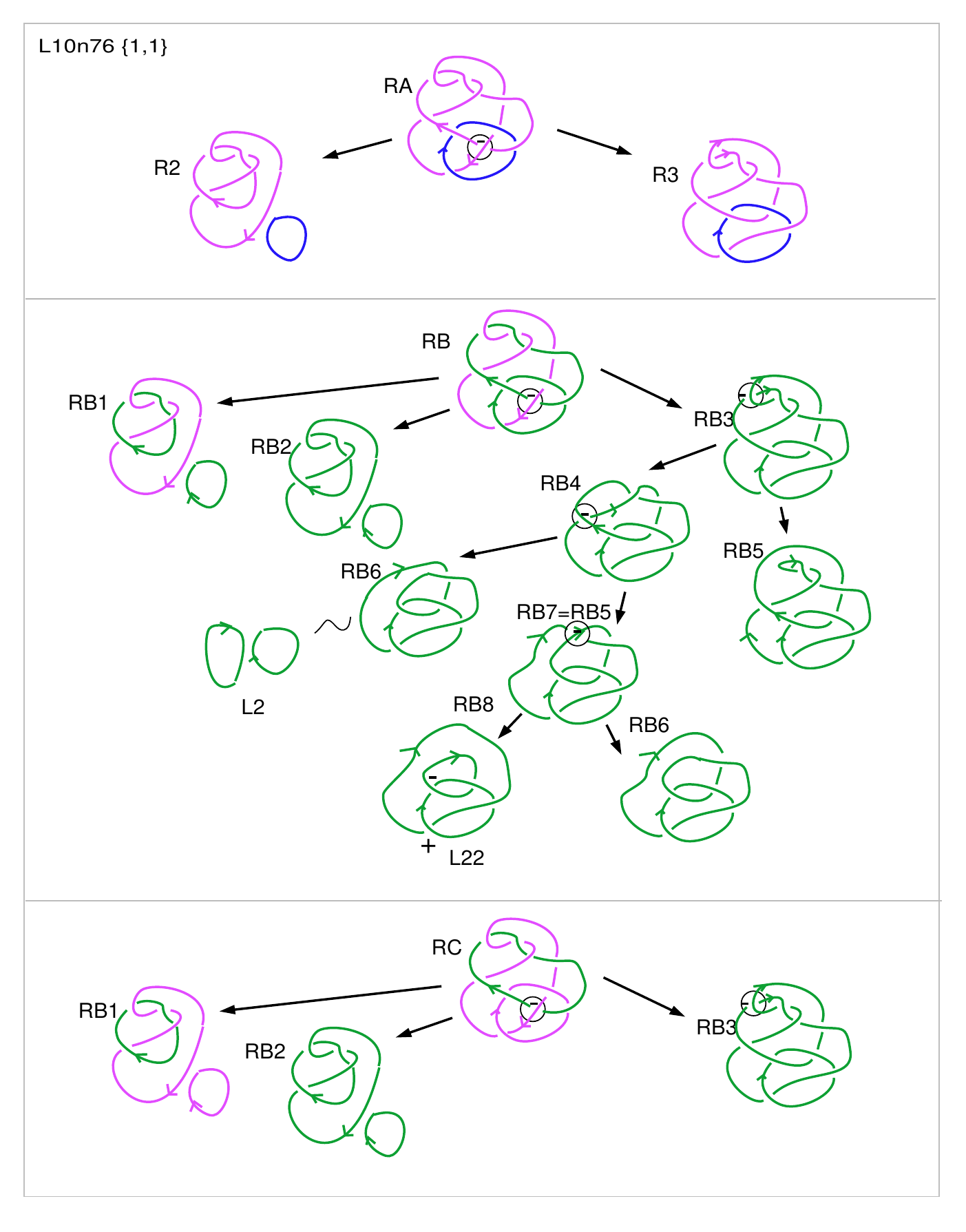}
\caption{Calculation of $\F^2_{(2,1)}$ on the  link   $\R$  }\label{RABC}
\end{figure}
\begin{figure}[H]
\centering
\includegraphics{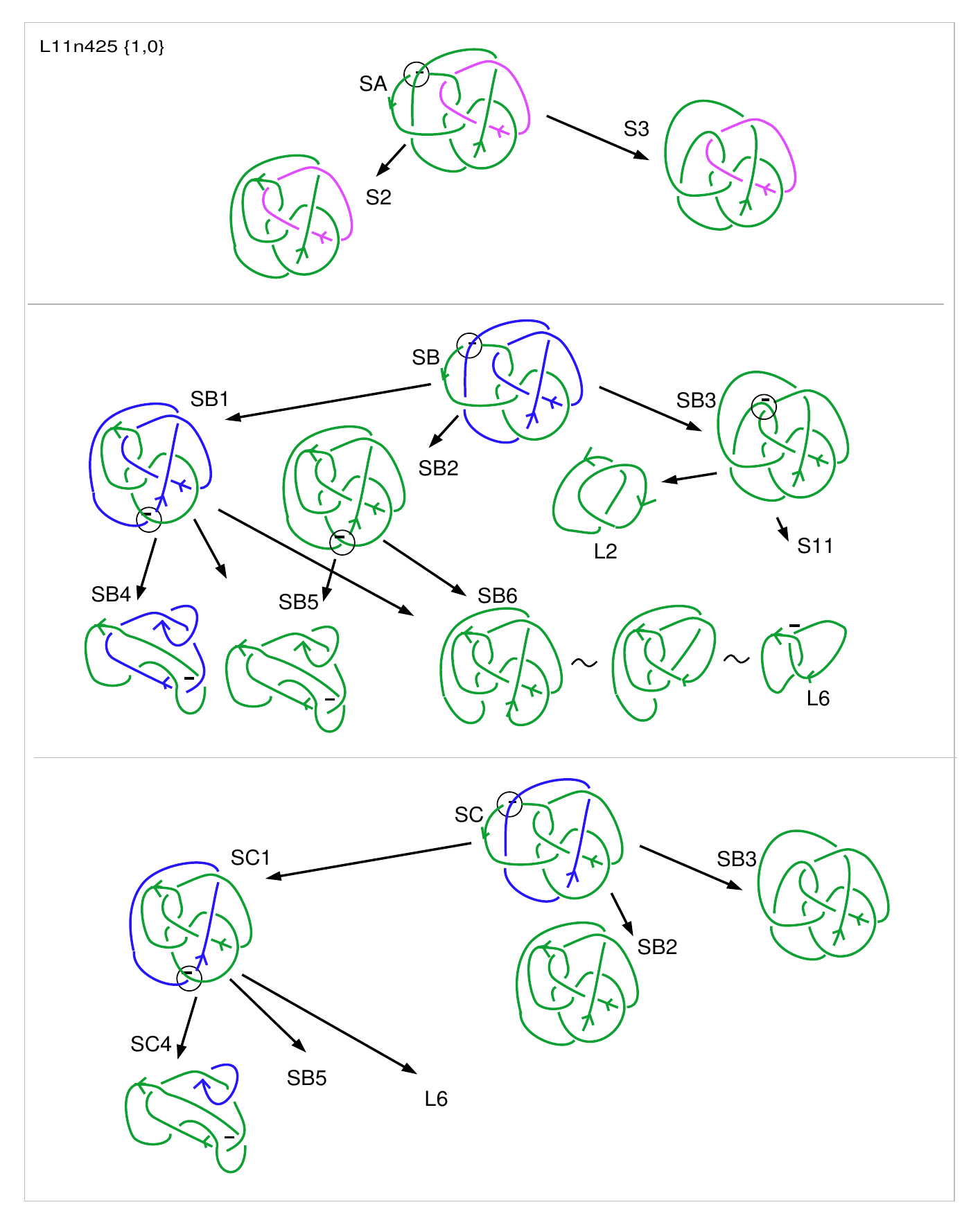}
\caption{Calculation of $\F^2_{(2,1)}$  on the  link   $\S$ }\label{SABC}
\end{figure}

\section*{Appendix 1:  List of values of $F$ on simple colored  links }
\begin{figure}[H]
\centering
\includegraphics{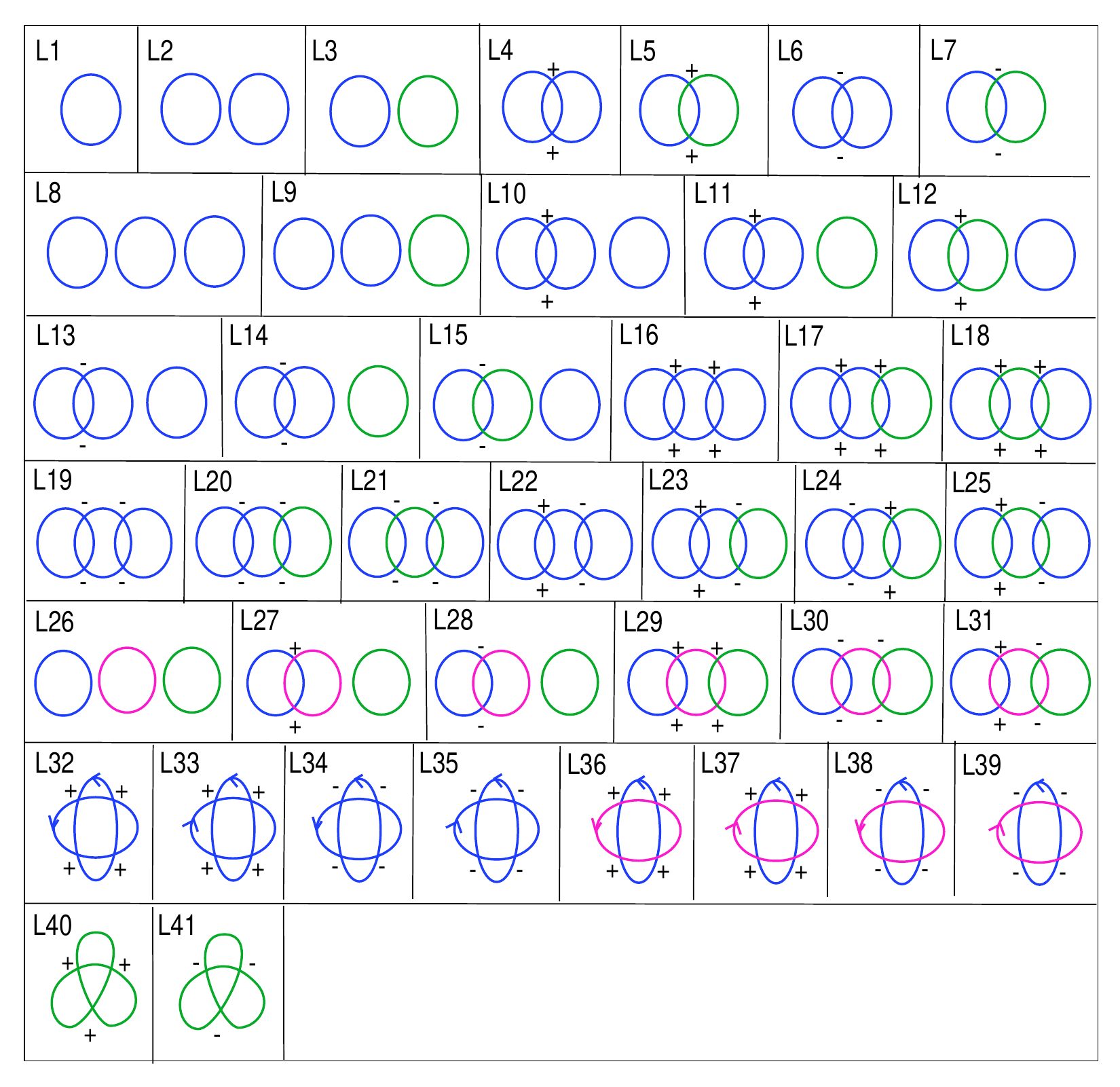}
\caption{Table of  simple colored links }\label{List}
\end{figure}
\vfill \newpage

\[ \begin{array} {|l|c|l|c|}
   \hline
       & F &   & F \\ \hline
   \L1 & 1 & \L2 & \frac{ tw^2-1}{(1-t)w} \\ \hline
   \L3 & \frac {1}{w x} & \L4 &  \frac{ w(t^2 w^2+t-t^2-1)}{ 1-t } \\ \hline
   \L5 &  \frac{(1-x t w^2+x t)w}{x} & \L6 &  \frac{t w^2+1-t^2 w^2-w^2}{t w^3 (t-1)} \\ \hline
   \L7 &  \frac{t-x+x w^2}{t w^3 x} & \L8 &  \frac{ (t w^2-1)^2}{ (1-t)^2 w^2}  \\ \hline
   \L9 &\frac{t w^2-1}{(1-t) w^2 x} & \L10 &  \frac{ (t^2 w^2+t-t^2-1) (t w^2-1)}{  (t-1)^2} \\  \hline
    \L11  &   \frac{t^2 w^2+t-t^2-1}{(1-t) x} & \L12 &  \frac{   (1-x t w^2+ x t) (t w^2-1)}{ x (1-t)} \\  \hline
    \L13  & \frac{  (t w^2+1-t^2 w^2-w^2)(1-t w^2)}{t w^4(t-1)^2}& \L14 & \frac{ t w^2+1-t^2 w^2-w^2}{  t w^4(t-1) x} \\  \hline
    \L15  & \frac{(t-x+x w^2) (t w^2-1)}{  t w^4 x (1-t)} &\L16 & \frac{w^2 (t^2 w^2+t-t^2-1)^2}{(t-1)^2} \\  \hline
    \L17  & \frac{ w^2 ( x t w^2-x t-1) (t^2 w^2+t-t^2-1)}{x (t-1)} &\L18 & \frac{w^2 (1-t w^2 +x t^2 w^4-2 x t w^2+2 x t
     +x t^3 w^4-2 x t^3 w^2-x t^2+x t^3)}{ x (t-1)} \\  \hline
    \L19  & \frac{(-t w^2-1+t^2 w^2+w^2)^2}{ (t^2 w^6 (t-1)^2} &\L20 &  \frac{ (t-x+x w^2) ( t w^2+1-t^2 w^2-w^2)}{ t^2 w^6 x (t-1)}\\  \hline
     \L21 &\frac {t^3 w^2-t^2-2 x t^2 w^2+x t+2 x t^2 w^4+x-x w^4 t-2 x w^2+x w^4}{ t^2 w^6 x (1-t)} & \L22 & \frac{ (-t w^2-1+t^2 w^2+w^2) (t^2 w^2+t-t^2-1)}{ (t-1)^2 w^2 t}\\  \hline
    \L23  &  \frac{(t-x+x w^2) (t^2 w^2+t-t^2-1)}{ w^2 x (1-t) t}&\L24 &\frac{ ( x t w^2-x t-1) (-t w^2-1+t^2 w^2+w^2)}{  w^2 x (t-1) t } \\  \hline
     \L25 & \frac{-t^2 w^2+t-x t w^2-x+x w^2-x t^2 w^4+x t^3 w^4+x t^2 w^2+x t-x t^3 w^2}{ w^2 x (t-1) t} &
    \L26  & \frac{1}{ w^2  x^2}  \\ \hline
\L27 &  \frac{ 1-x t w^2+ xt}{x^2} & \L28&  \frac{t-x+x w^2}{t w^4 x^2} \\ \hline
 \L29 & \frac{( x t w^2-x t-1)^2*w^2}{x^2} & \L30 &  \frac{(t-x+x w^2)^2}{ x^2 t^2 w^6}\\ \hline
   \L31 &  \frac{ ( 1-x t w^2+ x t) (t-x+x w^2)}{ w^2 t x^2} & \L32 &     \frac{ w^3(t^3 w^2+t^2-t^3-t-t^4 w^2+t^4-t^2 w^2+1)}{t-1}\\ \hline
   \L33 &   \frac{-w (t^3 w^4+t^2 w^2-t^3 w^2-t w^2-t^2+2 t-1)}{t-1}  &  \L34 & \frac{-(t^2 w^2+t-t^3 w^2-t w^2-t^2+t^4 w^2-1+w^2)}{ t^3 w^5 (t-1)} \\ \hline
     \L35 & \frac{t w^2+1-t^2 w^2-w^2-t^3 w^4+2 t^2 w^4-t w^4}{ t^2 w^5 (t-1)} & \L36 & \frac{ w^3 ( 1-x t w^2+x t-x t^3 w^2+x t^3)}{x} \\ \hline
 \L37 & \frac{w ( w^2-x w^4 t-x t^2 w^4+x t^2 w^2+x w^2+x t-x)}{x} & \L38 & \frac{t^3-x t^2+x t^2 w^2-x+x w^2}{x w^5 t^3} \\ \hline
   \L39 &   \frac{t^2-x t-x+x t^2 w^2+x w^2-x t^2 w^4+x w^4 t}{x w^5 t^2}  & \L40 &  t^2 w^2+w^2 -t^2 w^4 \\ \hline
   \L41 &   \frac{ t^2 w^2+w^2-1}{t^2 w^4} &  &  \\ \hline
 \end{array} \]
\vfill \newpage
\section*{Appendix 2   }

\[\begin{array} {| l|}
\hline
F^3( L11n325 \{1,1\})=
\\ (-x t^6 w^4+x^2 t^6 w^6+x t^5 w^4-x^2 t^4-x t w^2-3 x^2 t w^2+6 x^2 t^2 w^2-8 x^2 t^3 w^2
+6 x^2 t^4 w^2-3 x^2 t^5 w^2+\\ x^2 t^6 w^2+x t^3 w^4+x t^2 w^2-2 x t^3 w^2-2 x^2 t^2 w^4
+5 x^2 t^3 w^4-6 x^2 t^4 w^4+5 x^2 t^5 w^4-2 x^2 t^6 w^4-x^2+\\ x t+  x t^4 w^2-x t^5 w^2 +
  x t^6 w^2-x t^2-4 x^2 t^2+3 x^2 t-2 x^2 t^5 w^6+x^2 t^4 w^6+  x^2 w^2+3 x^2 t^3+t^5 w^2 +\\ x t^3-x t^4)/(w^6 x^2 t^5) \\
  F^3( L11n424 \{0,0\})= \\-(2 x t^6 w^4-2 x t^5 w^4+x^2 t^4+2 x^2 t w^2-7 x^2 t^2 w^2+8 x^2 t^3 w^2
   -5 x^2 t^4 w^2+4 x^2 t^5 w^2-x^2 t^6 w^2+\\2 x t^3 w^2+  3 x^2 t^2 w^4-4 x^2 t^3 w^4
   +6 x^2 t^4 w^4-6 x^2 t^5 w^4+x^2 t^6 w^4+x^2-x t^4 w^2+2 x t^5 w^2-2 x t^6 w^2+
  \\  4 x^2 t^2- 2 x^2 t-x t^4 w^4+2 x^2 t^5 w^6-2 x^2 t^4 w^6-x^2 w^2-4 x^2 t^3
   -t^5 w^2-2 x t^3+2 x t^4)/(w^6 x^2 t^5)  \\ \hline  \hline
F^3(L11n356 \{1,0\})= \\
(-t^2 x^2+2 t x^2-x t^5 w^4+x t^3 w^2-x^2+8 t^3 x^2 w^4-x^2 t^6 w^4+t^2 w^2 +x^2 t^6 w^6-4 t^2 x^2 w^4-x t^3 w^4-\\5 t^3 x^2 w^2-3 x^2 t^5 w^6-4 x^2 t^4 w^4 +x^2 t^4 w^2+3 x^2 t^5 w^4+2 x^2 w^2+x t^5 w^2+5 x^2 t^2 w^2-5 x^2 t w^2-\\x^2 w^4
-x t w^2+x t w^4-3 x^2 t^3 w^6+3 t x^2 w^4+3 x^2 t^4 w^6)/(w^2 x^2 t^2)\\
F^3(L11n434\{0,0\})=\\
 (x t^3-2 t^2 x^2+2 t x^2-x t+x+x t^3 w^2+8 t^3 x^2 w^4+x t^4 w^2-x^2 t^6 w^4-x t^4 w^4
+t^2 w^2+x^2 t^6 w^6-\\3 t^2 x^2 w^4- 2 x t^3 w^4-4 t^3 x^2 w^2-2 x^2 t^5 w^6-5 x^2 t^4 w^4
+2 x^2 t^4 w^2+2 x^2 t^5 w^4+x^2 w^2-x t^2+\\5 x^2 t^2 w^2-x w^2- 6 x^2 t w^2-x^2 w^4
+x t w^4-4 x^2 t^3 w^6+x t^2 w^4+4 t x^2 w^4+3 x^2 t^4 w^6)/(w^2 x^2 t^2) \\ \hline   \hline
F^3(L10n79\{1,1\})= \\
 -(-2 x^2 t^2 w^4+x^2 t^3 w^4+3 t^4 w^2 x^2-2 x^2 t^4 w^4+2 x^2 t^3-3 w^2 x^2 t^3
+t^7 x^2 w^4-x^2 t^4+x^2 t^6+\\t^7 x w^2+x^2 t^5-2 t^5 x w^2+2 t^5 x-t^7 x-x^2 t w^2
-x^2 t^7 w^2+x^2 w^2-3 x^2 t^5 w^2-x^2-t^6+2 x^2 t^5 w^4+\\5 w^2 x^2 t^2-2 x t^3 w^2
+2 t^3 x-3 x^2 t^2+x^2 t-x^2 t^6 w^2)/(w^8 x^2 t^6) \\
F^3(L10n95\{1,0\})= \\
(t^6-3 w^2 x^2 t^2+x^2 t^2 w^4-2 x^2 t+2 x^2 t w^2-2 x^2 t^5 w^4+2 x^2 t^2+x^2 t^4
+t^5 x w^2-2 x^2 t^3+x t^3 w^2-\\t^3 x-t^5 x+2 x^2 t^4 w^4+2 x^2 t^5 w^2+4 w^2 x^2 t^3
+x^2 t^6 w^4-2 x^2 t^3 w^4-x^2 w^2-x^2 t^6 w^2+x t w^2-\\ x t+x^2-3 t^4 w^2 x^2)/(w^8 x^2 t^6) \\ \hline    \hline
F^3(L11a404\{1,1\})=\\
-(x t^2 w^4-x w^2+x-2 x t^6 w^4-2 x^2 t^7 w^2+x^2 t^8 w^2+4 x^2 t^7 w^4-2 x^2 t^8 w^4
+2 x^2 t^6 w^6-2 x^2 t^7 w^6+\\x^2 t^8 w^6-4 x^2 t^4+x t w^2-3 x^2 t w^2+4 x^2 t^2 w^2
-15 x^2 t^3 w^2+9 x^2 t^4 w^2-15 x^2 t^5 w^2+6 x^2 t^6 w^2-\\x t^3 w^4- 4 x t^2 w^2
+x t^3 w^2-t^5-x^2 t^2 w^4+8 x^2 t^3 w^4-6 x^2 t^4 w^4+ 15 x^2 t^5 w^4-7 x^2 t^6 w^4
-\\x t-x t^4 w^2+ 4 x t^6 w^2-t^3+3 x t^2-3 x^2 t^2+2 x^2 t+2 x t^4 w^4-x^2 t^3 w^6\
+x^2 t w^4-4 x^2 t^5 w^6 + \\x^2 t^4 w^6+8 x^2 t^3+t^5 w^2+ 4 x^2 t^5-x^2 t^6
-2 x t^6-x t^4)/(x^2 t^3) \\
 F^3(L11a428 \{0,1\})=\\
  -(-t^6 x-3 x^2 t^4+t^4 x+4 x^2 t^5+7 x^2 t^3-t^3+t^5 x w^2-x^2+t^5 w^2-4 x^2 t^2+2 t^2 x+9 x^2 t^4 w^2+2 x^2 t+\\
  x^2 t^7+6 x^2 t^3 w^4-t^8 x+x^2 w^2-t^5-x^2 t^8 w^4-2 x^2 t^2 w^4-7 x^2 t^4 w^4-16 x^2 t^5 w^2-3 x^2 t^7 w^6+\\
  4 x^2 t^6 w^2+2 x^2 t^6 w^6+6 x^2 t^2 w^2+2 t^6 x w^2-x t^6 w^4+x^2 t^8 w^6+x^2 t^4 w^6-6 x^2 t^6 w^4+2 x t^4 w^4-\\
  13 x^2 t^3 w^2+16 x^2 t^5 w^4-4 x^2 t^7 w^2-4 x^2 t^5 w^6-t^5 x+6 x^2 t^7 w^4-3 t^4 x w^2+t^8 x w^2-t^7 x w^4-\\
  2 x^2 t w^2+t^7 x w^2-2 x t^2 w^2)/(x^2 t^3) \\ \hline \hline
  F^3(L11n358\{1,1\})= \\
 (-x^2 t+x^2 t^2-x^2 t^3+x^2 t^4+2 x^2 t w^2-2 x^2 t^2 w^2+2 x^2 t^2 w^4-x^2 t w^4
+3 x^2 t^3 w^2-3 x^2 t^4 w^2-\\4 x^2 t^3 w^4+2 x^2 t^3 w^6+2 x^2 t^4 w^4-x^2 t^2 w^6
-4 x^2 t^5 w^4+x^2 t^4 w^8+x^2 t^5 w^6+3 x^2 t^5 w^2-\\x^2 t^4 w^6-t^6 w^8 x^2
+2 t^6 w^4 x^2-t^6 w^2 x^2+t^7 w^6 x^2-t^7 w^4 x^2-x+t w^2+2 x t^3 w^4+x t
-\\2 x t w^2+x w^2+4 x t^2 w^2-4 x t^2 w^4-x t^3 w^2+x w^4 t+x w^2 t^4-w^6 x t^3
+x w^6 t^4-2 x w^4 t^4) w^2/(x^2 t) \\
F^3(L11n418\{1,0\})= \\
  (-2 x^2 t+x^2 t^2-x^2 t^3+x^2 t^4+3 x^2 t w^2-2 x^2 t^2 w^2+2 x^2 t^2 w^4
+4 x^2 t^3 w^2-3 x^2 t^4 w^2-5 x^2 t^3 w^4+\\ x^2 t^3 w^6+x^2 t^4 w^4-x^2 t^2 w^6
-4 x^2 t^5 w^4+x^2 t^5 w^6+3 x^2 t^5 w^2+x^2 t^4 w^6+ x^2-t^6 w^8 x^2 \\
+2 t^6 w^4 x^2-t^6 w^2 x^2-2 x^2 w^2+t^7 w^6 x^2-t^7 w^4 x^2+w^4 x^2-w^6 x t^2
+t w^2-2 x t^3 w^4-x t+2 x t w^2-\\ x w^2+x t^2+x t^2 w^2-x t^2 w^4+x t^3 w^2+x w^4
-x w^4 t+w^6 x t^3+ w^8 x^2 t^3-  w^6 x^2 t) w^2/(x^2 t)
\\ \hline   
\end{array}
\]

\section*{Appendix 3   }

\centerline{  $\F^3_{(1,1,1)}( L11n358\{0,1\})=\F^3_{(1,1,1)}( L11n418\{0,0\} ) $  }

$F(\K)=F(\H)= (-t x^2 w^4+x^2 t^5 w^2+x t^2 w^2+t^3 x^2+t^5-3 t^4 x-
x w^4 t^6-t^2 w^6 x^2-
t^4 x^2+t^5 x+x^2 t^7 w^6-x^2 t^7 w^4+x t^6 w^2+x w^4 t^7-x t^7 w^2+t^2 x^2-
2 x^2 w^2-x t^2+x^2 w^4+2 x^2 t^3 w^6+x^2 t^6 w^2-3 x^2 t^4 w^6+x^2 t w^2-
5 x^2 t^2 w^2+5 x^2 t^4 w^4+2 x^2 t^5 w^6+x^2+t^3 x^2 w^2+5 t^2 x^2 w^4-
4 t^3 x^2 w^4-x^2 t^6 w^6+3 x t^4 w^2-x^2 t^4 w^2-x t^5 w^2-
3 x^2 t^5 w^4)/(t^5 x^2 w^8)$

\vskip 1 cm

\centerline{ \bf $\F^2_{(2,1)}( L11n358\{0,1\})\not=\F^2_{(2,1)}( L11n418\{0,0\} ) ) $  }

\vskip .5 cm

$\F^2_{(2,1)}( L11a435\{0,0\})= \{ F(\K\A), F(\K\B), F(\K\C) \}$

$F(\K\A)= (2 t^3 w^2+t^4 w^4-4 t^4 w^2-2 t^3+6 x t^2 w^2-2 t^5+5 t^5 w^2+t^4 x-2 t^5 w^4+
 2 t^6 w^4-4 t^6 w^2+3 t^4-3 w^6 x t^3+5 w^6 x t^4+w^6 x t^2-6 x w^6 t^5+4 x w^6 t^6-
 4 x w^4 t^6-t^5 x+2 t^7 w^2-t^8 w^2+t^8 w^4-2 t^7 w^4+x w^4 t^7+x t^7 w^2-
 2 x t^7 w^6-t^2 w^2+x t-x t^3-x t^2+2 x w^2-x+t^6+7 x t^5 w^4+t^2-5 x t^3 w^2-
 x w^4+3 x t^4 w^2+9 x t^3 w^4-3 x t w^2+2 x t w^4-9 x t^4 w^4-6 x t^2 w^4+
 x t^8 w^6-x t^8 w^4)/((t-1) w^8 t^5 x),$

$F(\K\B)= -(-7 x t^2 w^2-t^5+t^4 x+t^6 w^2+3 w^6 x t^3-5 w^6 x t^4-w^6 x t^2+5 x w^6 t^5-
3 x w^6 t^6+3 x w^4 t^6-x w^4 t^7+x t^7 w^6-x t+2 x t^2-2 x w^2+x-5 x t^5 w^4+
5 x t^3 w^2+x w^4-5 x t^4 w^2-8 x t^3 w^4+3 x t w^2-2 x t w^4+9 x t^4 w^4+
6 x t^2 w^4)/((t-1) w^8 t^5 x), $

$F(\K\C) = -(t^3 w^2+t^4 w^4-t^4 w^2-t^3-7 x t^2 w^2+t^5 w^2-2 t^5 w^4+t^6 w^4+3 w^6 x t^3-
5 w^6 x t^4-w^6 x t^2+4 x w^6 t^5-2 x w^6 t^6+2 x w^4 t^6-x w^4 t^7+x t^7 w^6-x t-
x t^3+2 x t^2-2 x w^2+x-6 x t^5 w^4+6 x t^3 w^2+x w^4-4 x t^4 w^2-8 x t^3 w^4+
2 x t^5 w^2+3 x t w^2-2 x t w^4+9 x t^4 w^4+6 x t^2 w^4)/((t-1) w^8 t^5 x).$

\vskip 1 cm

$\F^2_{(2,1)}( L11n418\{0,0\} )= \{ F(\H\A),F(\H\B),F(\H\C) \}$

$F(\H\A) = (-t^4 w^2+7 x t^2 w^2-t^5+2 t^5 w^2+t^6 w^4-3 t^6 w^2+t^4-3 w^6 x t^3+5 w^6 x t^4+
w^6 x t^2-5 x w^6 t^5+3 x w^6 t^6-3 x w^4 t^6-t^5 x+2 t^7 w^2-t^8 w^2+t^8 w^4-
2 t^7 w^4+x w^4 t^7+x t^7 w^2-2 x t^7 w^6+x t-2 x t^2+2 x w^2-x+t^6+6 x t^5 w^4-
5 x t^3 w^2-x w^4+4 x t^4 w^2+8 x t^3 w^4-3 x t w^2+2 x t w^4-9 x t^4 w^4-6 x t^2 w^4+
x t^8 w^6-x t^8 w^4)/(w^8 t^5 x (t-1)),$

$F(\H\B)= -(-t^3 w^2+t^4 w^2+t^3-6 x t^2 w^2-t^5 w^2+t^6 w^2-t^4+3 w^6 x t^3-5 w^6 x t^4-
w^6 x t^2+5 x w^6 t^5-3 x w^6 t^6+3 x w^4 t^6-x w^4 t^7+x t^7 w^6+t^2 w^2-x t+
x t^2-2 x w^2+x-6 x t^5 w^4-t^2+6 x t^3 w^2+x w^4-4 x t^4 w^2-9 x t^3 w^4+
x t^5 w^2+3 x t w^2-2 x t w^4+9 x t^4 w^4+6 x t^2 w^4)/(w^8 t^5 x (t-1)),$

$F(\H\C)=  -(t^4 w^2-7 x t^2 w^2-t^5 w^2+t^6 w^2-t^4+3 w^6 x t^3-5 w^6 x t^4-w^6 x t^2+
5 x w^6 t^5-3 x w^6 t^6+3 x w^4 t^6-x w^4 t^7+x t^7 w^6-x t+2 x t^2-2 x w^2+
x-6 x t^5 w^4+5 x t^3 w^2+x w^4-4 x t^4 w^2-8 x t^3 w^4+x t^5 w^2+3 x t w^2-
2 x t w^4+9 x t^4 w^4+6 x t^2 w^4)/(w^8 t^5 x (t-1)).$

\vfill
\newpage

\centerline{  $\F^3_{(1,1,1)}( L11a467\{0,1\} )=\F^3_{(1,1,1)}( L11a527\{0,0\}) $  }

 $F(\G)=F(\Q)=  -(3 x t^4-x^2 t^2-x^2 t^3+3 x^2 w^2+9 x^2 t^2 w^2+14 x^2 t^3 w^4+5 x^2 t w^4-
  10 x^2 t^3 w^6-4 x^2 t^3 w^2-3 x t^4 w^2-13 x^2 t^4 w^4+13 x^2 t^4 w^6+2 x^2 t^4 w^2+
  w^8 x^2 t^3-3 t^4 w^8 x^2+t^7 x^2 w^4-t^6 x^2 w^2-2 t^7 w^6 x^2+x t^7 w^2-t^5+
  7 t^5 x^2 w^4-10 t^5 x^2 w^6-t^6 x^2 w^4+5 t^6 x^2 w^6+t^5 x w^2+t^6 w^4 x-
  t^7 x w^4-t^6 x w^2+4 w^8 t^5 x^2-x^2 w^6 t-t^5 x^2 w^2+t^7 w^8 x^2-3 t^6 w^8 x^2+
  x^2 t^4-x t^5-x^2-x t^2 w^2+x t^2+5 w^6 x^2 t^2-4 w^2 t x^2-2 x^2 w^4-
  13 w^4 x^2 t^2)/(w^8 t^5 x^2)$

\vskip 1 cm

\centerline{ \bf $\F^2_{(2,1)}( L11a467\{0,1\} )\not=\F^2_{(2,1)} (L11a527\{0,0\} )$  }

\vskip .5 cm
$\F^2_{(2,1)}( L11a467\{0,1\} )=\{F(\G\A), F(\G\B),  F(\G\C)\}$

$ F( \G\A) = (-8 x t^4 w^2+27 x t^4 w^4-23 x t^4 w^6+15 t^3 w^6 x-6 t^2 w^6 x-7 w^8 t^5 x-
18 t^5 x w^4+22 t^5 w^6 x+t^4 w^4-2 t^5 w^4+x w^6 t+3 t^5 x w^2-4 t^7 w^8 x+
7 t^6 w^8 x+4 t^4 w^8 x-14 t^6 w^6 x+7 t^6 w^4 x-2 t^7 x w^4-t^3 w^8 x+t^5 w^2+
t^6 w^4+x-26 x t^3 w^4-3 x w^2+2 x w^4-x t^8 w^6+x t^8 w^8+7 x t w^2+6 t^7 w^6 x-
x t-14 x t^2 w^2+2 x t^2+t^3 w^2-x t^3+18 x t^2 w^4+13 x t^3 w^2-7 x w^4 t-
t^4 w^2-t^3)/(x t^5 w^8 (1-t)),$

$F(\G\B) = (-x t^4-7 x t^4 w^2+27 x t^4 w^4-23 x t^4 w^6+15 t^3 w^6 x-6 t^2 w^6 x-x t^7 w^2-
7 w^8 t^5 x+2 t^5-19 t^5 x w^4+24 t^5 w^6 x-t^4 w^4+2 t^5 w^4+x w^6 t+t^5 x w^2-
4 t^7 w^8 x+7 t^6 w^8 x+4 t^4 w^8 x-16 t^6 w^6 x+9 t^6 w^4 x-2 t^7 x w^4-t^3 w^8 x-
5 t^5 w^2+2 t^7 w^4-2 t^7 w^2-2 t^6 w^4+4 t^6 w^2+x t^5-t^8 w^4-t^6+t^8 x w^4+x-
27 x t^3 w^4-3 x w^2+t^2 w^2-t^2+2 x w^4-2 x t^8 w^6+x t^8 w^8+7 x t w^2+7 t^7 w^6 x-
x t-13 x t^2 w^2+x t^2-2 t^3 w^2+x t^3+18 x t^2 w^4+12 x t^3 w^2-7 x w^4 t+4 t^4 w^2+
t^8 w^2-3 t^4+2 t^3)/(x t^5 w^8 (1-t)),$

$F(\G\C) = (x t^4-9 x t^4 w^2+27 x t^4 w^4-23 x t^4 w^6+15 t^3 w^6 x-6 t^2 w^6 x-7 w^8 t^5 x-
t^5-17 t^5 x w^4+23 t^5 w^6 x+x w^6 t+t^5 x w^2-4 t^7 w^8 x+7 t^6 w^8 x+4 t^4 w^8 x-
15 t^6 w^6 x+8 t^6 w^4 x-2 t^7 x w^4-t^3 w^8 x+t^6 w^2+x-26 x t^3 w^4-3 x w^2+
2 x w^4-x t^8 w^6+x t^8 w^8+7 x t w^2+6 t^7 w^6 x-x t-14 x t^2 w^2+2 x t^2+
18 x t^2 w^4+12 x t^3 w^2-7 x w^4 t)/(x t^5 w^8 (1-t)).$

 \vskip 1 cm
$\F^2_{(2,1)}(L11a527\{0,0\} )=\{ F(\Q\A), F(\Q\B),  F(\Q\C) \}$

$F(\Q\A) = (-8 x t^4 w^2+27 x t^4 w^4-23 x t^4 w^6+15 t^3 w^6 x-6 t^2 w^6 x-7 w^8 t^5 x-
18 t^5 x w^4+23 t^5 w^6 x+x w^6 t+2 t^5 x w^2-4 t^7 w^8 x+7 t^6 w^8 x+4 t^4 w^8 x-
15 t^6 w^6 x+8 t^6 w^4 x-2 t^7 x w^4-t^3 w^8 x-t^5 w^2+t^6 w^2+x-26 x t^3 w^4-
3 x w^2+2 x w^4-x t^8 w^6+x t^8 w^8+7 x t w^2+6 t^7 w^6 x-x t-14 x t^2 w^2+
2 x t^2+18 x t^2 w^4+12 x t^3 w^2-7 x w^4 t+t^4 w^2-t^4)/(x t^5 w^8 (1-t)),$

$F(\Q\B)= (-8 x t^4 w^2+27 x t^4 w^4-23 x t^4 w^6+15 t^3 w^6 x-6 t^2 w^6 x-x t^7 w^2-
7 w^8 t^5 x+t^5-18 t^5 x w^4+23 t^5 w^6 x+x w^6 t+t^5 x w^2-4 t^7 w^8 x+7 t^6 w^8 x+
4 t^4 w^8 x-15 t^6 w^6 x+8 t^6 w^4 x-2 t^7 x w^4-t^3 w^8 x-2 t^5 w^2+2 t^7 w^4-
2 t^7 w^2-t^6 w^4+3 t^6 w^2+x t^5-t^8 w^4-t^6+t^8 x w^4+x-26 x t^3 w^4-3 x w^2+
2 x w^4-2 x t^8 w^6+x t^8 w^8+7 x t w^2+7 t^7 w^6 x-x t-14 x t^2 w^2+2 x t^2+
18 x t^2 w^4+12 x t^3 w^2-7 x w^4 t+t^4 w^2+t^8 w^2-t^4)/(x t^5 w^8 (1-t)),$

$F(\Q\C) = (-8 x t^4 w^2+27 x t^4 w^4-23 x t^4 w^6+15 t^3 w^6 x-6 t^2 w^6 x-7 w^8 t^5 x-
18 t^5 x w^4+23 t^5 w^6 x+x w^6 t+2 t^5 x w^2-4 t^7 w^8 x+7 t^6 w^8 x+4 t^4 w^8 x-
15 t^6 w^6 x+8 t^6 w^4 x-2 t^7 x w^4-t^3 w^8 x-t^5 w^2+t^6 w^2+x-27 x t^3 w^4-
3 x w^2+t^2 w^2-t^2+2 x w^4-x t^8 w^6+x t^8 w^8+7 x t w^2+6 t^7 w^6 x-x t-
13 x t^2 w^2+x t^2-t^3 w^2+18 x t^2 w^4+13 x t^3 w^2-7 x w^4 t+t^4 w^2-
t^4+t^3)/(x t^5 w^8 (1-t)).$

\vfill \newpage
\section*{Appendix 4  }
\centerline{$\F^3_{(1,1,1)}$  and  $\F^2_{(2,1)}$ on  the  links  $\R=L10n76\{1,1\}$ and  $\S=L11n425\{1,0\}$.}

$\F^3_{(1,1,1)}(\R)= \{F^3(\R) \}$

$F^3(\R) = (-x^2 w^4 t+t^3 w^6 x^2+3 x^2 t^2 w^4-4 x^2 t^3 w^4+t^5 x w^2-t^5 x w^4-x^2 t^4 w^2
+3 x^2 t^3 w^2-x^2 t^5 w^4-4 x^2 t^2 w^2+x^2 t^5 w^6+3 x^2 t^4 w^4-2 x^2 t^4 w^6
+t^5-t^4 x+x^2 t^2-2 x^2 t-x^2 w^2+x^2+3 w^2 t x^2-x t+x t w^2-x t^2 w^2+t^4 w^4 x
+x t^2 w^4-x t^3 w^4+2 x t^3 w^2-x t^3)/(t^5 x^2 w^8).$

\vskip .5 cm

$\F^2_{(2,1)}(\R)= \{F(\R\A), F(\R\B), F(\R\C) \}$

$F(\R\A) = (-3 t^5 w^6 x-t^5 x w^2+4 t^5 x w^4+4 x t^4 w^2+2 t^5 w^4-t^5 w^2-t^6 w^4+t^3
+t^6 w^6 x-t^6 w^4 x-t w^2+t-x+2 t^4 w^2+3 x t-4 x t w^2+x w^2-t^2 w^4+2 t^3 w^4
+2 t^2 w^2-t^2+7 x t^2 w^2-3 x t^2-7 t^4 w^4 x-t^3 w^6 x+3 t^4 w^6 x-2 t^4 w^4
-4 x t^2 w^4+7 x t^3 w^4-7 x t^3 w^2-3 t^3 w^2+x t^3+x w^4 t)/(w^8 t^5 x (t-1)),$

 $F(\R\B)= (-4 t^5 w^6 x+t^5 x w^2+3 t^5 x w^4+5 x t^4 w^2+t^5-t^4 x-t^6 w^2+2 t^6 w^6 x
 -2 t^6 w^4 x-x+2 x t-3 x t w^2+x w^2+7 x t^2 w^2-3 x t^2-8 t^4 w^4 x-2 t^3 w^6 x
 +4 t^4 w^6 x-4 x t^2 w^4+7 x t^3 w^4-5 x t^3 w^2+x w^4 t)/(w^8 t^5 x (t-1)),$

  $F(\R\C)= (-4 t^5 w^6 x+t^5 x w^2+3 t^5 x w^4+5 x t^4 w^2+t^5-t^4 x-t^6 w^2+2 t^6 w^6 x
  -2 t^6 w^4 x-x+2 x t-3 x t w^2+x w^2+7 x t^2 w^2-3 x t^2-8 t^4 w^4 x-2 t^3 w^6 x
  +4 t^4 w^6 x-4 x t^2 w^4+7 x t^3 w^4-5 x t^3 w^2+x w^4 t)/(w^8 t^5 x (t-1)).$

\vskip 1 cm
$\F^3_{(1,1,1)}(\S)= \{F^3(\S) \}$

$F^3(\S) = -(x^2 w^4 t-2 t^3 w^6 x^2-3 x^2 t^2 w^4+4 x^2 t^3 w^4-2 t^5 x w^2+2 t^5 x w^4
+x^2 t^4 w^2-2 x^2 t^3 w^2+x^2 t^5 w^4+5 x^2 t^2 w^2-4 x^2 t^4 w^4+2 x^2 t^4 w^6
-t^5+2 t^4 x-2 x^2 t^2+x^2 t+x^2 w^2-x^2-t^6 x+t^5 x^2+t^4 x^2-t^6 x^2 w^2
-2 t^5 x^2 w^2+t^6 x^2 w^4+t^6 x w^2-2 w^2 t x^2-2 t^4 w^4 x-2 x t^3 w^2
+2 x t^3)/(t^5 x^2 w^8).$

\vskip .5 cm
$\F^2_{(2,1)}(\S)= \{F(\S\A), F(\S\B), F(\S\C) \}$

$F(\S\A)= (-3 t^5 w^6 x-t^5 x w^2+4 t^5 x w^4+4 x t^4 w^2+2 t^5 w^4-t^5 w^2-t^6 w^4+t^3
+t^6 w^6 x-t^6 w^4 x-x+t^4 w^2+2 x t-3 x t w^2+x w^2+7 x t^2 w^2-3 x t^2-8 t^4 w^4 x
-2 t^3 w^6 x+4 t^4 w^6 x-t^4 w^4-4 x t^2 w^4+7 x t^3 w^4-6 x t^3 w^2-t^3 w^2
+x t^3+x w^4 t)/(w^8 t^5 x (t-1)),$

$F(\S\B)= -(4 t^5 w^6 x-t^5 x w^2-3 t^5 x w^4-5 x t^4 w^2-t^5+t^4 x-2 t^6 w^6 x+2 t^6 w^4 x
+t^6 x+x-t^6 x w^2-t^7+t^6+t^7 w^2-t^7 x w^2+t^7 x w^4-2 x t+3 x t w^2-x w^2
-7 x t^2 w^2+3 x t^2+8 t^4 w^4 x+2 t^3 w^6 x-4 t^4 w^6 x+4 x t^2 w^4-7 x t^3 w^4
+5 x t^3 w^2-x w^4 t)/(w^8 t^5 x (t-1)),$

$F(\S\C)= (-3 t^5 w^6 x-t^5 x w^2+4 t^5 x w^4+4 x t^4 w^2+2 t^5 w^4-t^5 w^2-t^6 w^4+t^3
+t^6 w^6 x-t^6 w^4 x-x+t^4 w^2+2 x t-3 x t w^2+x w^2+7 x t^2 w^2-3 x t^2-8 t^4 w^4 x
-2 t^3 w^6 x+4 t^4 w^6 x-t^4 w^4-4 x t^2 w^4+7 x t^3 w^4-6 x t^3 w^2-t^3 w^2+x t^3
+x w^4 t)/(w^8 t^5 x (t-1))$.

\end{document}